\documentclass{amsart}
\usepackage{paralist}
\usepackage{url}
\usepackage{mathrsfs}
\usepackage{color}
\usepackage{graphicx}
\usepackage{caption}
\usepackage{subcaption}
\usepackage{multirow}
\usepackage{mathtools}
\usepackage{algorithm}
\usepackage{algorithmic}
\usepackage{tikz}
\usepackage{csvsimple}
\usepackage{booktabs}
\usepackage{xcolor}

\newcommand{\CC}{{\mathbb C}}
\newcommand{\bC}{{\mathbb C}}
\newcommand{\bE}{{\mathbb E}}

\newcommand{\cF}{{\mathcal F}}
\newcommand{\cI}{{\Omega}}
\newcommand{\bB}{{\mathbb B}}

\newcommand{\re}{\mathbb{R}}

\newcommand{\cpx}{\mathbb{C}}

\newcommand{\N}{\mathbb{N}}

\newcommand{\diag}{\mbox{diag}}

\newcommand{\lmd}{\lambda}

\newcommand{\eps}{\epsilon}

\def\af{\alpha}

\def\rank{\mbox{rank}}

\newcommand{\Sig}{\Sigma}

\newcommand{\reff}[1]{(\ref{#1})}

\newcommand{\mc}[1]{\mathcal{#1}}

\newcommand{\bdes}{\begin{description}}
    \newcommand{\edes}{\end{description}}
\newcommand{\bal}{\begin{align}}
\newcommand{\eal}{\end{align}}
\newcommand{\bnum}{\begin{enumerate}}
    \newcommand{\enum}{\end{enumerate}}
\newcommand{\bit}{\begin{itemize}}
    \newcommand{\eit}{\end{itemize}}
\newcommand{\bea}{\begin{eqnarray}}
\newcommand{\eea}{\end{eqnarray}}
\newcommand{\be}{\begin{equation}}
\newcommand{\ee}{\end{equation}}
\newcommand{\baray}{\begin{array}}
    \newcommand{\earay}{\end{array}}
\newcommand{\bsry}{\begin{subarray}}
    \newcommand{\esry}{\end{subarray}}
\newcommand{\bca}{\begin{cases}}
    \newcommand{\eca}{\end{cases}}
\newcommand{\bcen}{\begin{center}}
    \newcommand{\ecen}{\end{center}}
\newcommand{\bbm}{\begin{bmatrix}}
    \newcommand{\ebm}{\end{bmatrix}}
\newcommand{\bmx}{\begin{matrix}}
    \newcommand{\emx}{\end{matrix}}
\newcommand{\bpm}{\begin{pmatrix}}
    \newcommand{\epm}{\end{pmatrix}}
\newcommand{\btab}{\begin{tabular}}
    \newcommand{\etab}{\end{tabular}}

\newtheorem{theorem}{Theorem}[section]

\theoremstyle{definition}
\newtheorem{example}[theorem]{Example}

\newtheorem{alg}[theorem]{Algorithm}

\newtheorem{remark}[theorem]{Remark}

\numberwithin{equation}{section}

\begin{document}
\title[Gaussian Mixture Models and Incomplete Tensor Decompositions]
{Learning Diagonal Gaussian Mixture Models and Incomplete Tensor Decompositions}

\author[Bingni Guo]{Bingni~Guo}
\address{Bingni Guo, Jiawang Nie, and Zi Yang,
Department of Mathematics, University of California San Diego,
9500 Gilman Drive, La Jolla, CA, USA, 92093.}
\email{b8guo@ucsd.edu, njw@math.ucsd.edu, ziy109@ucsd.edu}

\author[Jiawang Nie]{Jiawang~Nie}

\author[Zi Yang]{Zi~Yang}



\begin{abstract}
This paper studies how to learn parameters in diagonal Gaussian mixture models.
The problem can be formulated as computing incomplete symmetric tensor decompositions.
We use generating polynomials
to compute incomplete symmetric tensor decompositions and approximations.
Then the tensor approximation method is used to learn diagonal Gaussian mixture models.
We also do the stability analysis.
When the first and third order moments are sufficiently accurate,
we show that the obtained parameters
for the Gaussian mixture models are also highly accurate.
Numerical experiments are also provided.
\end{abstract}

\maketitle

\section{Introduction}

A Gaussian mixture model consists of several component Gaussian distributions.
For given samples of a Gaussian mixture model, people often need to estimate parameters
for each component Gaussian distribution \cite{ge2018parameter,lee2005effective}.
Consider a Gaussian mixture model with $r$ components.
For each $i\in [r]:=\{1,\ldots,r\}$,
let $\omega_i$ be the positive probability for the
$i$th component Gaussian to appear in the mixture model.
We have each $\omega_i>0$ and $\sum_{i=1}^r\omega_i=1$.
Suppose the $i$th Gaussian distribution is
$\mathcal{N}(\mu_i,\Sigma_i)$, where $\mu_i\in\mathbb{R}^d$ is the expectation (or mean) and $\Sigma_i\in\mathbb{R}^{d\times d}$ is the covariance matrix.
Let $y\in\mathbb{R}^d$ be the random vector for the Gaussian mixture model and let
$y_1, \ldots, y_N$ be identically independent distributed (i.i.d)
samples from the mixture model.
Each $y_j$ is sampled from one of the $r$ component Gaussian distributions,
associated with a label $Z_j\in[r]$ indicating the component that it is sampled from.
The probability that a sample comes from the $i$th component is $\omega_i$.
When people observe only samples without labels,
the $Z_j$'s are called latent variables.
The density function for the random variable $y$ is
\begin{equation*}
f(y) \,:=\,   \sum_{i=1}^r\omega_i
\frac{1}{\sqrt{(2\pi)^d \det \Sigma_i} }
\exp\Big\{-\frac{1}{2}(y-\mu_i)^T \Sigma_i^{-1} (y-\mu_i)  \Big\},
\end{equation*}
where $\mu_i$ is the mean and $\Sigma_i$ is the covariance matrix
for the $i$th component.

Learning a Gaussian mixture model is to estimate the parameters
$\omega_i,\mu_i,\Sigma_i$ for each $i \in [r]$, from given samples of $y$.
The number of parameters in a covariance matrix grows quadratically
with respect to the dimension.
Due to the curse of dimensionality, the computation becomes very expensive
for large $d$ \cite{magdon2010approximating}.
Hence, diagonal covariance matrices are preferable in applications.
In this paper, we focus on learning Gaussian mixture models
with diagonal covariance matrices, i.e.,
\[
\Sigma_i \,= \, \text{diag}\big(\sigma_{i1}^2,\ldots,\sigma_{id}^2\big),
\quad i=1,\ldots, r.
\]
A natural approach for recovering the unknown parameters $\omega_i,\mu_i,\Sigma_i$
is the method of moments.
It estimates parameters by solving a system of multivariate polynomial equations,
from moments of the random vector $y$.
Directly solving polynomial systems may encounter non-existence or non-uniqueness of
statistically meaningful solutions \cite{wu2020optimal}.
However, for diagonal Gaussians, the third order moment tensor
can help us avoid these troubles.

Let $M_3 := \mathbb{E} (y \otimes y \otimes y)$
be the third order tensor of moments for $y$. One can write that
$y=\eta(z)+\zeta(z)$, where $z$ is a discrete random variable such that
$\mbox{Prob}(z=i)=\omega_i$, $\eta(i)=\mu_i \in \re^d$ and $\zeta(i)$
is the random variable $\zeta_i$ obeying the Gaussian distribution $\mathcal{N}(0,\Sigma_i)$.
Assume all $\Sigma_i$ are diagonal, then
\begin{multline}
M_3 =\sum_{i=1}^{r} \omega_i \mathbb{E}[(\eta(i)+\zeta_i)^{\otimes  3}] =
\sum_{i=1}^r \omega_i \Big(
    \mu_i\otimes\mu_i\otimes\mu_i+
    \mathbb{E}[\mu_i\otimes \zeta_i\otimes \zeta_i]+ \\
    \mathbb{E}[\zeta_i\otimes \mu_i\otimes \zeta_i]
    +\mathbb{E}[\zeta_i\otimes \zeta_i\otimes \mu_i] \Big).
\end{multline}
The second equality holds because $\zeta_i$ has zero mean and
\[
\mathbb{E}[\zeta_i\otimes \zeta_i\otimes \zeta_i]=
\mathbb{E}[\mu_i\otimes \mu_i\otimes \zeta_i]=
\mathbb{E}[\zeta_i\otimes \mu_i\otimes \mu_i]=
\mathbb{E}[\mu_i\otimes \zeta_i\otimes \mu_i]=0.
\]
The random variable $\zeta_i$ has diagonal covariance matrix,
so $\mathbb{E}[(\zeta_i)_j(\zeta_i)_l] = 0 $ for $j\neq l$. Therefore,
\begin{equation*}
\sum_{i=1}^r\omega_i\mathbb{E}[\mu_i\otimes \zeta_i\otimes \zeta_i]
= \sum_{i=1}^r\sum\limits_{j=1}^d
\omega_i\sigma_{ij}^2\mu_i\otimes e_j\otimes e_j=
\sum\limits_{j=1}^d a_j\otimes e_j\otimes e_j,
\end{equation*}
where the vectors $a_j$ are given by
\be \label{expr:aj}
a_j \coloneqq\sum_{i=1}^r\omega_i\sigma^2_{ij}\mu_i, \quad j=1,\ldots,d.
\ee
Similarly, we have
\[
\sum_{i=1}^r\omega_i\mathbb{E}[\zeta_i\otimes\mu_i\otimes \zeta_i]
=\sum_{j=1}^d e_j\otimes a_j\otimes e_j, \,
\sum_{i=1}^r\omega_i\mathbb{E}[\zeta_i\otimes \zeta_i\otimes\mu_i]
=\sum_{j=1}^d e_j\otimes e_j\otimes a_j.
\]
Therefore, we can express $M_3$ in terms of $\omega_i, \mu_i, \Sigma_i$ as
 \begin{equation}\label{M3-decomp}
 	M_3=\sum\limits_{i=1}^{r} \omega_i\mu_i\otimes\mu_i\otimes\mu_i
 + \sum\limits_{j=1}^d \Big( a_j\otimes e_j\otimes e_j+e_j\otimes a_j\otimes e_j
 +e_j\otimes e_j\otimes a_j \Big).
\end{equation}
We are particularly interested in the following third order symmetric tensor
\begin{equation} \label{def:tensor:F}
\cF := \sum_{i=1}^r \omega_i\mu_i\otimes\mu_i\otimes\mu_i.
\end{equation}
When the labels $i_1,i_2,i_3$ are distinct from each other, we have
\[
(M_3)_{i_1i_2i_3} \, = \, (\cF)_{i_1i_2i_3} \quad \text{for} \quad
i_1 \ne i_2 \ne i_3 \ne i_1.
\]
Denote the label set
\begin{equation}\label{eq:label set}
 \cI\:= \{ (i_1, i_2, i_3): \,  i_1 \ne i_2 \ne i_3 \ne i_1,\,
i_1,i_2,i_3 \, \text{are labels for} \, M_3 \}.
\end{equation}
The tensor $M_3$ can be estimated from the samplings for $y$,
so the entries $\cF_{i_1i_2i_3}$ with $(i_1,i_2,i_3) \in \cI$
can also be obtained from the estimation of $M_3$.
To recover the parameters $\omega_i, \mu_i$,
we first find the tensor decomposition for $\cF$,
from the partially given entries
$\cF_{i_1i_2i_3}$ with $(i_1,i_2,i_3) \in \cI$.
Once the parameters $\omega_i, \mu_i$ are known,
we can determine $\Sigma_i$ from the expressions of $a_j$ as in \reff{expr:aj}.

The above observation leads to the incomplete tensor decomposition problem.
For a third order symmetric tensor $\cF$ whose partial entries $\cF_{i_1i_2i_3}$
with $(i_1,i_2,i_3) \in \cI$ are known,
we are looking for vectors $p_1,\ldots, p_r$ such that
\be \label{iSTD:F}
\cF_{i_1i_2i_3} \, = \,
\Big( p_1^{\otimes 3} + \cdots + p_r^{\otimes 3}  \Big)_{i_1i_2i_3},
\quad \text{for all} \quad  (i_1,i_2,i_3) \in \cI.
\ee
The above is called an incomplete tensor decomposition for $\cF$.
To find such a tensor decomposition for $\cF$,
a straightforward approach is to do tensor completion:
first find unknown tensor entries $\cF_{i_1i_2i_3}$ with $(i_1,i_2,i_3) \not\in \cI$
such that the completed $\cF$ has low rank, and then compute the tensor decomposition for $\cF$.
However, there are serious disadvantages for this approach.
The theory for tensor completion or recovery, especially for symmetric tensors,
is premature. Low rank tensor completion or recovery is typically not
guaranteed by the currently existing methodology.
Most methods for tensor completion are based on convex relaxations,
e.g., the nuclear norm or trace minimization
\cite{FriLim18,MHWG,njwSTNN17,TanSha15,YuaZha16}.
These convex relaxations may not produce low rank completions \cite{NIPS20135077}.

In this paper, we propose a new method for determining incomplete tensor decompositions.
It is based on the generating polynomial method in \cite{nie2017generating}.
The label set $\cI$ consists of $(i_1,i_2,i_3)$ of distinct $i_1,i_2,i_3$.
We can still determine some generating polynomials, from the partially given tensor entries
$\cF_{i_1i_2i_3}$ with $(i_1,i_2,i_3) \in \cI$.
They can be used to get the incomplete tensor decomposition.
We show that this approach works very well when the rank $r$
is roughly not more than half of the dimension $d$.
Consequently, the parameters for the Gaussian mixture model can be recovered from
the incomplete tensor decomposition of $\cF$.

\medskip \noindent
{\bf Related Work} \,
Gaussian mixture models have broad applications in machine learning problems,
e.g., automatic speech recognition
\cite{karpagavalli2016review,povey2011subspace,reynolds1995speaker},
hyperspectral unmixing problem \cite{bioucas2012hyperspectral, ma2019hyperspectral},
background subtraction \cite{zivkovic2004improved,lee2005effective}
and anomaly detection \cite{veracini2009fully}.
They also have applications in social and biological sciences \cite{haas2006modelling,shekofteh2015gaussian,zhang2007probabilistic}.

There exist methods for estimating unknown parameters for Gaussian mixture models.
A popular method is the expectation-maximization (EM)
algorithm that iteratively approximates the maximum likelihood parameter estimation \cite{dempster1977maximum}.
This approach is widely used in applications,
while its convergence property
is not very reliable \cite{redner1984mixture}.
Dasgupta \cite{dasgupta1999learning} introduced a method
that first projects data to a randomly chosen
low-dimensional subspace and then use the empirical means and covariances
of low-dimensional clusters to estimate the parameters.
Later, Arora and Kannan \cite{sanjeev2001learning}
extended this idea to arbitrary Gaussians.
Vempala and Wong \cite{vempala2004spectral}
introduced the spectral technique to enhance the separation condition
by projecting data to principal components of the sample matrix
instead of selecting a random subspace.
For other subsequent work, we refer to
Dasgupta and Schulman \cite{dasgupta2013two},
Kannan et al. \cite{kannan2005spectral},
Achlioptas et al. \cite{achlioptas2005spectral},
Chaudhuri and Rao \cite{chaudhuri2008learning},
Brubaker and Vempala \cite{brubaker2008isotropic}
and Chaudhuri et al. \cite{chaudhuri2009multi}.

Another frequently used approach is based on moments, introduced by
Pearson~\cite{pearson1894contributions}.
Belkin and Sinha \cite{belkin2009learning} proposed
a learning algorithm for identical spherical Gaussians $(\Sigma_i=\sigma^2 I)$
with arbitrarily small separation between mean vectors.
It was also shown in \cite{kalai2010efficiently} that a mixture of two Gaussians
can be learned with provably minimal assumptions.
Hsu and Kakade \cite{hsu2013learning}
provided a learning algorithm for a mixture of spherical Gaussians, i.e.,
each covariance matrix is a multiple of the identity matrix.
This method is based on moments up to order three
and only assumes non-degeneracy instead of separations.
For general covariance matrices, Ge et al.~\cite{ge2015learning}
proposed a learning method when the dimension $d$ is sufficiently high.
More moment-based methods for general latent variable models can be found in \cite{latent}.

\medskip  \noindent
{\bf Contributions} \,
This paper proposes a new method for learning diagonal Gaussian mixture models,
based on samplings for the first and third order moments.
Let $y_1,\cdots,y_N$ be samples and let
$\{(\omega_i,\mu_i,\Sigma_i):i\in[r]\}$
be parameters of the diagonal Gaussian mixture model,
where each covariance matrix $\Sigma_i$ is diagonal.
We use the samples $y_1,\cdots,y_N$ to estimate the third order moment tensor $M_3$,
as well as the mean vector $M_1$.
We have seen that the tensor $M_3$ can be expressed as in \reff{M3-decomp}.

For the tensor $\cF$ in \reff{def:tensor:F}, we have
$\cF_{i_1i_2i_3} = (M_3)_{i_1i_2i_3}$ when the labels $i_1,i_2,i_3$
are distinct from each other. Other entries of $\cF$ are not known,
since the vectors $a_j$ are not available.
The $\cF$ is an incompletely given tensor.
We give a new method for computing the incomplete tensor decomposition of
$\cF$ when the rank $r$ is low
(roughly no more than half of the dimension $d$).
The tensor decomposition of $\cF$ is unique
under some genericity conditions \cite{COV17},
so it can be used to recover parameters $\omega_i,\mu_i$.
To compute the incomplete tensor decomposition of $\cF$,
we use the generating polynomial method in \cite{nie2017generating,nie2017low}.
We look for a special set of generating polynomials for $\cF$,
which can be obtained by solving linear least squares.
It only requires to use the known entries of $\cF$.
The common zeros of these generating polynomials
can be determined from eigenvalue decompositions. Under some genericity assumptions,
these common zeros can be used to get the incomplete tensor decomposition.
After this is done, the parameters $\omega_i,\mu_i$
can be recovered by solving linear systems.
The diagonal covariance matrices $\Sigma_i$ can also be
estimated by solving linear least squares.
The tensor $M_3$ is estimated from the samples $y_1, \ldots, y_N$.
Typically, the tensor entries $(M_3)_{i_1i_2i_3}$ and $\cF_{i_1i_2i_3}$,
are not precisely given. We also provide a stability analysis for this case,
showing that the estimated parameters are also accurate
when the entries $(M_3)_{i_1i_2i_3}$ have small errors.

The paper is organized as follows.
In Section~\ref{sc:pre}, we review some basic results
for symmetric tensor decompositions and generating polynomials.
In Section~\ref{sc:iSTD}, we give a new algorithm for computing
an incomplete tensor decomposition for $\cF$,
when only its subtensor $\cF_{\Omega}$ is known.
Section~\ref{sc:errSTD} gives the stability analysis
when there are errors for the subtensor $\cF_{\Omega}$.
Section~\ref{sc:learnDGM} gives the algorithm
for learning Gaussian mixture models.
Numerical experiments and applications are given in Section~\ref{sc:num}.
We make some conclusions and discussions in Section~\ref{sc:con}.

\section{Preliminary}
\label{sc:pre}

\subsection*{Notation}
Denote $\mathbb{N}$, $\mathbb{C}$ and $\mathbb{R}$ the set of
nonnegative integers, complex and real numbers respectively.
Denote the cardinality of a set $L$ as $|L|$. Denote by $e_i$
the $i$th standard unit basis vector, i.e., the $i$the entry of $e_i$ is one
and all others are zeros. For a complex number $c$, $\sqrt[n]{c}$ or $c^{1/n}$ denotes the principal $n$th root of $c$. For a complex vector $v$, $\text{Re}(v),\text{Im}(v)$ denotes the real part and imaginary part of $v$ respectively. A property is said to be generic if it is true in the whole space except a subset of zero Lebesgue measure.
The $\|\cdot\|$ denotes the Euclidean norm of a vector or the Frobenius norm of a matrix.
For a vector or matrix, the superscript $^T$ denotes the transpose
and $^H$ denotes the conjugate transpose. For $i,j \in \N$, $[i]$
denotes the set $\{1,2,\ldots,i\}$ and $[i,j]$ denotes the set $\{i,i+1,\ldots,j\}$ if $i \le j$.
For a vector $v$, $v_{i_1:i_2}$ denotes the vector $(v_{i_1},v_{i_1+1},\ldots,v_{i_2})$.
For a matrix $A$, denote by $A_{[i_1:i_2, j_1:j_2]}$ the submatrix of $A$
whose row labels are $i_1,i_1+1,\ldots,i_2$ and whose column labels are $j_1,j_1+1\ldots,j_2$.
For a tensor $\cF$, its subtensor $\cF_{[i_1:i_2,j_1:j_2,k_1:k_2]}$
is similarly defined.

Let $\textrm{S}^m(\bC^{d})$ (resp., $\textrm{S}^m(\re^{d})$) denote the space of
$m$th order symmetric tensors over the vector space $\bC^{d}$ (resp., $\re^{d}$).
For convenience of notation, the labels for tensors start with $0$.
A symmetric tensor $\mathcal{A} \in \textrm{S}^m(\bC^{n+1})$ is labelled as
\[
\mathcal{A} \, = \, (\mathcal{A}_{i_1...i_m} )_{ 0\le i_1, \ldots, i_m \le n} ,
\]
where the entry $\mathcal{A}_{i_1\ldots i_m}$ is invariant for all permutations of
$(i_1,\ldots,i_m)$. The Hilbert-Schmidt norm $\|\mathcal{A}\|$ is defined as
\begin{equation} \label{tensor:norm:HS}
\|\mathcal{A}\| \,:= \,
\Big(\sum\limits_{0\leq i_1,\ldots,i_m\leq n}|
\mathcal{A}_{i_1\ldots i_m}|^2\Big)^{1/2}.
\end{equation}
The norm of a subtensor $\| \mc{A}_{\Omega} \|$ is similarly defined.
For a vector $u:=(u_0,u_1,\ldots,u_n)\in \bC^{d}$, the tensor power
$u^{\otimes m}:=u\otimes \cdots \otimes u$, where $u$ is repeated $m$ times,
is defined such that
\[
(u^{\otimes m})_{i_1\ldots i_m} \, = \, u_{i_1}
    \times \cdots \times u_{i_m} .
\]
For a symmetric tensor $\mathcal{F}$, its symmetric rank is
\begin{equation*}
    \text{rank}_{\textrm{S}}(\mathcal{F})\coloneqq
    \text{min}\left\{r \, \mid \, \mathcal{F}=\sum\limits_{i=1}^r u_i^{\otimes m}\right\}.
\end{equation*}
There are other types of tensor ranks \cite{Land12,Lim13}.
In this paper, we only deal with symmetric tensors and symmetric ranks.
We refer to \cite{CLQY20,DeSLim08,Fri16,HLim13,Land12,Lim13}
for general work about tensors and their ranks.
For convenience, if $r=\text{rank}_{\textrm{S}}(\mathcal{F})$,
we call $\mathcal{F}$ a rank-$r$ tensor and
$\mathcal{F}=\sum\limits_{i=1}^r u_i^{\otimes m}$ is called a rank decomposition.

%
For a power $\alpha\coloneqq(\alpha_1,\alpha_2,\cdots,\alpha_{n})\in\mathbb{N}^{n}$
and $x\coloneqq(x_1,x_2,\cdots,x_{n})$, denote
\[
|\alpha|\coloneqq\alpha_1+\alpha_2+\cdots+\alpha_{n},\quad
x^{\alpha}\coloneqq x_1^{\alpha_1}x_2^{\alpha_2}\cdots x_{n}^{\alpha_{n}}, \quad
x_0 :=1.
\]
The monomial power set of degree $m$ is denoted as
\[
\mathbb{N}^n_m \coloneqq \{\alpha=(\alpha_1,\alpha_2,\cdots,\alpha_{n})
\in\mathbb{N}^n:|\alpha|\leq m\}.
\]
For $\alpha  \in \mathbb{N}_3^n $, we can write that
$x^\alpha = x_{i_1}x_{i_2}x_{i_3} $ for some $0\le i_1,i_2,i_3 \le n$.
%
%

Let $\mathbb{C}[x]_m$ be the space of all polynomials in $x$ with complex coefficients
and whose degrees are no more than $m$. For a cubic polynomial
$p\in\mathbb{C}[x]_3$ and $\mathcal{F}\in \textrm{S}^3(\mathbb{C}^{n+1})$,
we define the bilinear product (note that $x_0 = 1$)
\begin{equation}
    \langle p,\mathcal{F}\rangle=\sum\limits_{0\leq i_1,i_2,i_3\leq n}p_{i_1i_2i_3}\mathcal{F}_{i_1i_2i_3}
    \quad\text{for}\quad p \, =\sum\limits_{0\leq i_1,i_2,i_3\leq n}p_{i_1i_2i_3}x_{i_1}x_{i_2}x_{i_3},
\end{equation}
where $p_{i_1i_2i_3}$ are coefficients of $p$.
A polynomial $g\in\mathbb{C}[x]_3$ is called a {\it generating polynomial}
for a symmetric tensor $\mathcal{F} \in \textrm{S}^3(\bC^{n+1})$ if
\begin{equation}\label{gene_poly}
    \langle g\cdot x^{\beta}, \mathcal{F}\rangle=0\quad\forall\beta\in\mathbb{N}_{3-\text{deg}(g)}^n ,
\end{equation}
where $\text{deg}(g)$ denotes the degree of $g$ in $x$.
When the order is bigger than $3$, we refer to \cite{nie2017generating}
for the definition of generating polynomials. They can be used to compute
symmetric tensor decompositions and low rank approximations
\cite{nie2017generating,nie2017low}, which are closely related to
truncated moment problems and polynomial optimization
\cite{FiaNie12,NS09,nieloc,nie2015linear,Tight19}.
There are special versions of symmetric tensors and their decompositions
\cite{DNY20,NieYe19,NieYang20}.

\section{Incomplete Tensor Decomposition}
\label{sc:iSTD}

This section discusses how to compute an incomplete tensor decomposition
for a symmetric tensor $\cF \in \mathrm{S}^3(\cpx^d)$
when only its subtensor $\cF_\Omega$ is given,
for the label set $\Omega$ in \reff{eq:label set}.
For convenience of notation, the labels for $\cF$ begin with zeros
while a vector $u\in \bC^{d}$ is still labelled as $u:=(u_1,\ldots,u_d)$. We set
\[
n:=d-1, \quad x = (x_1, \ldots, x_n), \quad x_0 := 1.
\]
For a given rank $r$, denote the monomial sets
\begin{equation}  \label{basis}
\mathscr{B}_0\coloneqq\{x_{1},\cdots,x_{r}\},
\quad
\mathscr{B}_1=\{x_i x_j:\, i \in [r],\  j \in [r+1, n] \}.
\end{equation}
For a monomial power $\alpha \in \N^n$, by writing $\alpha\in\mathscr{B}_1$,
we mean that $x^{\alpha} \in \mathscr{B}_1$.
For each $\alpha \in \mathscr{B}_1$, one can write
$\alpha = e_i +e_j$ with $i \in [r],\  j \in [r+1, n]$.
Let $\CC^{[r] \times \mathscr{B}_1}$  denote the space of matrices
labelled by the pair $(k,\alpha)\in [r] \times \mathscr{B}_1$.
For each $\alpha = e_i + e_j\in\mathscr{B}_1$ and $G\in\CC^{[r] \times \mathscr{B}_1}$,
denote the quadratic polynomial in $x$
\begin{equation}\label{gp_def}
\varphi_{ij}[G](x) \, \coloneqq  \,
    \sum\limits_{k =1}^{r}G(k, e_i+e_j)x_k- x_i x_j.
\end{equation}

Suppose $r$ is the symmetric rank of $\cF$.
A matrix $G\in\CC^{[r] \times \mathscr{B}_1}$ is called a {\it generating matrix}
of $\cF$ if each $\varphi_{ij}[G](x)$, with $\alpha = e_i + e_j \in \mathscr{B}_1$,
is a generating polynomial of $\cF$.
Equivalently, $G$ is a generating matrix of $\cF$ if and only if
\begin{equation}\label{eq:generating}
  \langle x_t \varphi_{ij}[G](x) ,\cF \rangle
  =\sum_{k=1}^r G(k,e_i+e_j)\cF_{0kt}-\cF_{ijt}
  \,= \, 0, \quad t = 0, 1, \ldots, n,
\end{equation}
for all $i \in [r],\  j \in [r+1, n]$.
The notion {\it generating matrix} is motivated from the fact that
the entire tensor $\cF$ can be recursively determined by $G$
and its first $r$ entries (see \cite{nie2017generating}).
The existence and uniqueness of the generating matrix $G$ is shown as follows.

\begin{theorem} \label{thm:unique G}
Suppose $\cF$ has the decomposition
\begin{equation} \label{eq:decom_1 F}
\cF = \lambda_1\begin{bmatrix*}
            1 \\ u_1
        \end{bmatrix*}
        ^{\otimes 3}+\cdots+\lambda_r\begin{bmatrix*}
            1 \\ u_r
        \end{bmatrix*}^{\otimes 3} ,
\end{equation}
for vectors $u_i \in \CC^{n}$ and scalars $ 0\neq \lambda_i \in \cpx$.
If the subvectors $(u_1)_{1:r},\ldots,(u_r)_{1:r}$ are linearly independent,
then there exists a unique generating matrix
$G\in\CC^{[r] \times\mathscr{B}_1}$
satisfying \reff{eq:generating} for the tensor $\cF$.
\end{theorem}
\begin{proof}
We first prove the existence.
For each $i=1, \ldots, r$, denote the vectors $v_i = (u_i)_{1:r}$.
Under the given assumption, $V := [v_1 \, \ldots \, v_r]$
is an invertible matrix. For each $l=r+1,\ldots,n$, let
\be \label{Nl:eigdecomp}
N_l \, := \, V \cdot \diag\big( (u_1)_{l},\ldots,(u_r)_{l} \big) \cdot  V^{-1}.
\ee
Then $N_l v_i = (u_i)_l v_i$ for $i=1,\ldots,r$, i.e.,
$N_l$ has eigenvalues
$(u_1)_{l},\ldots,(u_r)_{l} $ with corresponding eigenvectors
$(u_1)_{1:r},\ldots,(u_r)_{1:r}$.
We select $G\in\CC^{[r] \times\mathscr{B}_1}$ to be the matrix such that
\begin{equation}
 N_l =  \begin{bmatrix}
        G(1,e_1+e_l) & \cdots & G(r,e_1+e_l) \\
        \vdots & \ddots & \vdots \\
        G(1,e_r+e_l) & \cdots & G(r,e_r+e_l) \\
        \end{bmatrix},\, l=r+1,\ldots,n.
\end{equation}
For each $s =1,\ldots,r$ and $\alpha = e_i + e_j \in \bB_1$
with $i \in [r],\  j \in [r+1, n]$,
\[
\varphi_{ij}[G](u_s) =\sum_{k=1}^r G(k,e_i+e_j)(u_s)_k -
(u_s)_i (u_s)_j \, = \, 0 .
\]
For each $t = 1,\ldots, n$, it holds that
\begin{eqnarray*}
\langle x_t \varphi_{ij}[G](x) ,\cF \rangle
&=& \left\langle \sum_{k =1}^{r}G(k, e_i+e_j)  x_tx_k   -   x_t x_i x_j, \cF \right\rangle \\
&=& \left\langle \sum_{k =1}^{r}G(k, e_i+e_j)  x_tx_k -  x_t x_i x_j,
       \sum_{s=1}^r \lmd_s \bbm 1 \\ u_s \ebm^{\otimes 3} \right\rangle  \\
&=& \sum_{k =1}^{r}G(k, e_i+e_j)\sum_{s=1}^r \lambda_s (u_s)_t (u_s)_k -
        \sum_{s=1}^r \lambda_s (u_s)_t (u_s)_i (u_s)_j \\
&=& \sum_{s=1}^r \lambda_s (u_s)_t  \left(\sum\limits_{k =1}^{r}G(k, e_i+e_j)(u_s)_k -(u_s)_i (u_s)_j \right)  \\
&=& 0.
\end{eqnarray*}
When $t=0$, we can similarly get
\begin{eqnarray*}
  \langle \varphi_{ij}[G](x) ,\cF \rangle
  &=& {\left\langle \sum_{k =1}^{r}G(k, e_i+e_j)  x_k   -   x_i x_j, \cF \right\rangle} \\
  &=& \sum_{s=1}^r \lambda_s {\left(\sum\limits_{k =1}^{r}G(k, e_i+e_j)(u_s)_k -(u_s)_i (u_s)_j \right)} \\
  &=& 0.
  \end{eqnarray*}
Therefore, the matrix $G$ satisfies \reff{eq:generating}
and it is a generating matrix for $\cF$.

Second, we prove the uniqueness of such $G$.
For each $\alpha = e_i + e_j \in \mathscr{B}_1$, let
\[
F := \begin{bmatrix*}
            \cF_{011} &  \cdots & \cF_{0r1} \\
            \vdots & \ddots & \vdots \\
            \cF_{01n} &  \cdots & \cF_{0rn}
        \end{bmatrix*},\,
g_{ij} :=
        \begin{bmatrix*}
             \cF_{1ij} \\ \vdots \\
            \cF_{nij}
        \end{bmatrix*} .
\]
Since $G$ satisfies \reff{eq:generating}, we have $F \cdot G(:,e_i+e_j) = g_{ij}$.
The decomposition \reff{eq:decom_1 F} implies that
\[
F = \bbm u_1 & \cdots & u_r \ebm \cdot
\mbox{diag}(\lambda_1,\ldots,\lambda_r)
 \cdot \bbm v_1 & \cdots & v_r \ebm^T.
\]
The sets $\{v_1,\ldots,v_r \}$ and $\{ u_1,\ldots,u_r \}$
are both linearly independent. Since each $\lmd_i \ne 0$,
the matrix $F$ has full column rank.
Hence, the generating matrix $G$ satisfying $F \cdot G(:,e_i+e_j) = g_{ij}$
for all $i \in [r], j\in [r+1,n]$ is unique.
\end{proof}

The following is an example of generating matrices.

\begin{example}\label{ex-1}
Consider the tensor $\cF\in\mathtt{S}^3(\CC^6)$ that is given as
\[
\cF = 0.4\cdot(1,1,1,1,1,1)^{\otimes 3} +
0.6\cdot(1,-1,2,-1,2,3)^{\otimes 3}.
\]
The rank $r=2$, $\mathscr{B}_0=\{x_1,x_2\}$ and
$
\mathscr{B}_1 \,=\, \{x_1x_3,x_1x_4,x_1x_5,x_2x_3,x_2x_4,x_2x_5\} .
$	
We have the vectors
\[
u_1 =(1,1,1,1,1), \quad u_2 = (-1,2,-1,2,3), \quad
v_1 =(1,1), \quad v_2 = (-1,2).
\]
The matrices $N_3$, $N_4$, $N_5$ as in \reff{Nl:eigdecomp} are
\begin{align*}
N_3
&= \begin{bmatrix}1 & -1 \\ 1 & 2\end{bmatrix}
\begin{bmatrix}1 & 0 \\ 0 & -1\end{bmatrix}
\begin{bmatrix}1 & -1 \\ 1 & 2\end{bmatrix}^{-1}=
\begin{bmatrix}1/3 & 2/3 \\ 4/3 & -1/3\end{bmatrix}, \\
N_4
&= \begin{bmatrix}1 & -1 \\ 1 & 2\end{bmatrix}
\begin{bmatrix}1 & 0 \\ 0 & 2\end{bmatrix}
\begin{bmatrix}1 & -1 \\ 1 & 2\end{bmatrix}^{-1}=
\begin{bmatrix}4/3 & -1/3 \\ -2/3 & 5/3\end{bmatrix}, \\
N_5
&= \begin{bmatrix}1 & -1 \\ 1 & 2\end{bmatrix}
\begin{bmatrix}1 & 0 \\ 0 & 3\end{bmatrix}
\begin{bmatrix}1 & -1 \\ 1 & 2\end{bmatrix}^{-1}=
\begin{bmatrix}5/3 & -2/3 \\ -4/3 & 7/3\end{bmatrix} .
\end{align*}
The entries of the generating matrix $G$ are listed as below:
\begin{equation}\label{example-G}
 \begin{array}{c|rrrrrr}
k\backslash(i,j) &(1,3) &(1,4) & (1,5)& (2,3)&(2,4) & (2,5) \\ \hline
	1	&  1/3 & 4/3  & 5/3  & 4/3 & -2/3  &  -4/3\\
	2	&  2/3 & -1/3 & -2/3 & -1/3  & 5/3 & 7/3
	\end{array} .
\end{equation}
The generating polynomials in \reff{gp_def} are
	\begin{equation*}
	\begin{aligned}
		\varphi_{13}[G](x) &= \frac{1}{3}x_1+\frac{2}{3}x_2-x_1x_3,\\
		\varphi_{14}[G](x) &= \frac{4}{3}x_1-\frac{1}{3}x_2-x_1x_4,\\
		\varphi_{15}[G](x) &= \frac{5}{3}x_1-\frac{2}{3}x_2-x_1x_5,
	\end{aligned}
	\quad
	\begin{aligned}
		\varphi_{23}[G](x) &= \frac{4}{3}x_1-\frac{1}{3}x_2-x_2x_3,\\
		\varphi_{24}[G](x) &= -\frac{2}{3}x_1+\frac{5}{3}x_2-x_2x_4,\\
		\varphi_{25}[G](x) &= -\frac{4}{3}x_1+\frac{7}{3}x_2-x_2x_5.
	\end{aligned}
  \end{equation*}
Above generating polynomials can be written in the following form
\[
\begin{bmatrix}
     \varphi_{1j}[G](x)\\
     \varphi_{2j}[G](x)
\end{bmatrix} = N_j\begin{bmatrix}
    x_1 \\ x_2
  \end{bmatrix} - x_j \begin{bmatrix}
    x_1\\ x_2
  \end{bmatrix},\, \text{for } j=3,4,5.
\]
For $x$ to be a common zero of $\varphi_{1j}[G](x)$ and $\varphi_{2j}[G](x)$,
it requires that $(x_1,x_2)$ is an eigenvector of $N_j$
with the corresponding eigenvalue $x_j$.
\end{example}

\subsection{Computing the tensor decomposition}
\label{subsec: decom}

We show how to find an incomplete tensor decomposition \reff{eq:decom_1 F} for
$\cF$ when only its subtensor $\cF_{\Omega}$ is given,
where the label set $\Omega$ is as in \reff{eq:label set}.
Suppose that there exists the decomposition \reff{eq:decom_1 F} for $\cF$,
for vectors $u_i \in \CC^{n}$ and nonzero scalars $\lambda_i \in \bC $.
Assume the subvectors $(u_1)_{1:r},\ldots,(u_r)_{1:r} $ are linearly independent,
so there is a unique generating matrix $G$ for $\cF$, by Theorem~\ref{thm:unique G}.

For each $\alpha = e_i + e_j \in \mathscr{B}_1$
with $i \in [r], j \in [r+1,n]$ and for each
\[ l=r+1,\ldots,j-1,j+1,\ldots,n , \]
the generating matrix $G$ satisfies the equations
\begin{equation} \label{eq: generateT}
\left\langle x_l \left( \sum_{k =1}^{r}G(k, e_i+e_j)x_k - x_i x_j \right),\cF \right\rangle
 = \sum_{k =1}^{r}G(k, e_i+e_j) \cF_{0kl} - \cF_{ijl} = 0 .
\end{equation}
Let the matrix $A_{ij}[\cF]\in \CC^{(n-r-1)\times r} $
and the vector $b_{ij}[\cF]\in \CC^{n-r-1}$ be such that
\begin{equation}\label{eq:A,b}
    A_{ij}[\cF] := \begin{bmatrix}
        \cF_{0,1,r+1} & \cdots & \cF_{0,r,r+1} \\
        \vdots & \ddots & \vdots \\
        \cF_{0,1,j-1} & \cdots & \cF_{0,r,j-1} \\
        \cF_{0,1,j+1} & \cdots & \cF_{0,r,j+1} \\
        \vdots & \ddots & \vdots \\
        \cF_{0,1,n} & \cdots & \cF_{0,r,n}
    \end{bmatrix}, \quad
    b_{ij}[\cF] :=  \begin{bmatrix}
        \cF_{i,j,r+1}\\
        \vdots \\
        \cF_{i,j,j-1}\\
        \cF_{i,j,j+1}\\
        \vdots \\
        \cF_{i,j,n}
    \end{bmatrix} .
\end{equation}
To distinguish changes in the labels of tensor entries of $\cF$,
the commas are inserted to separate labeling numbers.

The equations in \reff{eq: generateT} can be equivalently written as
\begin{equation} \label{solve-G}
 A_{ij}[\cF] \cdot G(:, e_i+e_j) \,= \, b_{ij}[\cF].
\end{equation}
If the rank $r\le \frac{d}{2}-1$, then $n-r-1 = d-r-2\ge r$. Thus, the number of rows is not less than the number of columns for matrices $A_{ij}[\cF]$. If $A_{ij}[\cF]$ has linearly independent columns,
then \reff{solve-G} uniquely determines $G(:,\alpha)$.
For such a case, the matrix $G$ can be fully determined
by the linear system \reff{solve-G}.
Let $N_{r+1}(G),\ldots,N_m(G) \in \bC^{r\times r}$ be the matrices given as
\begin{equation}\label{eq:N}
N_l(G) =  \begin{bmatrix}
        G(1,e_1+e_l) & \cdots &G(r,e_1+e_l) \\
        \vdots & \ddots & \vdots \\
        G(1,e_r+e_l) & \cdots &G(r,e_r+e_l) \\
    \end{bmatrix},\, l=r+1,\ldots, n .
\end{equation}
As in the proof of Theorem \ref{thm:unique G}, one can see that
\be \label{eigeq:NlGvi=uilvi}
    N_l(G) \begin{bmatrix}
        (u_i)_1 \\ \vdots \\(u_i)_r
    \end{bmatrix} = (u_i)_l \cdot
    \begin{bmatrix}
        (u_i)_1 \\ \vdots \\(u_i)_r
    \end{bmatrix}, \, l=r+1,\ldots, n  .
\ee
The above is equivalent to the equations
\[
N_l(G)v_i \, = \, (w_i)_{l-r} \cdot v_i, \quad  l=r+1,\ldots, n,
\]
for the vectors ($i=1,\ldots,r$)
\be \label{vecs:viwi}
v_i \,:=\, (u_i)_{1:r}, \quad w_i \,:=\, (u_i)_{r+1:n} .
\ee
Each $v_i$ is a common eigenvector of the matrices
$N_{r+1}(G),\ldots,N_n(G)$ and $(w_i)_{l-r}$
is the associated eigenvalue of $N_l(G)$.
These matrices may or may not have repeated eigenvalues.
Therefore, we select a generic vector $\xi\coloneqq(\xi_{r+1},\cdots,\xi_{n})$ and let
\begin{equation}\label{N_xi}
    N(\xi) \, \coloneqq  \, \xi_{r+1}N_{r+1}+\cdots+\xi_{n}N_{n} .
\end{equation}
The eigenvalues of $N(\xi)$ are $\xi^Tw_1, \ldots, \xi^Tw_r$.
When $w_1, \ldots, w_r$ are distinct from each other
and $\xi$ is generic, the matrix $N(\xi)$ does not have a repeated eigenvalue
and hence it has unique eigenvectors $v_1,\ldots,v_r$, up to scaling.
Let $\tilde{v}_1,\ldots,\tilde{v}_r$ be unit length eigenvectors of $N(\xi)$.
They are also common eigenvectors of $N_{r+1}(G)$, $\ldots$, $N_n(G)$.
For each $i = 1, \ldots,r$, let $\tilde{w}_i$ be the vector such that its $j$th entry
$(\tilde{w}_i)_j$ is the eigenvalue of $N_{j+r}(G)$,
associated to the eigenvector $\tilde{v}_i$, or equivalently,
\begin{equation}\label{w_i}
    \tilde{w}_i=(\tilde{v}_i^H N_{r+1}(G)\tilde{v}_i,\cdots,
    \tilde{v}_i^H N_n(G)\tilde{v}_i)\quad i=1,\ldots,r.
\end{equation}
Up to a permutation of $(\tilde{v}_1,\ldots, \tilde{v}_r)$,
there exist scalars $\gamma_i$ such that
\be \label{vi=tdvi:wi=tdwi}
v_i = \gamma_i \tilde{v}_i, \quad w_i = \tilde{w}_i .
\ee
The tensor decomposition of $\cF$ can also be written as
\[
\cF = \lambda_1
\begin{bmatrix}
        1 \\ \gamma_1 \tilde{v}_1 \\ \tilde{w}_1
\end{bmatrix}^{\otimes 3} + \cdots +\lambda_r
\begin{bmatrix}
        1 \\ \gamma_r \tilde{v}_r \\ \tilde{w}_r
\end{bmatrix}^{\otimes 3}.
\]
The scalars $\lambda_1,\cdots,\lambda_r$ and $\gamma_1,\cdots,\gamma_r$
satisfy the linear equations
\begin{equation*}
\baray{rcl}
    \lambda_1\gamma_1 \tilde{v}_1 \otimes\tilde{w}_1 +\cdots+
     {\lambda_r}{\gamma_r} \tilde{v}_r \otimes \tilde{w}_r
     &=& \cF_{[0,1:r,r+1:n]}, \\
    {\lambda_1}{\gamma_1^2}\tilde{v}_1\otimes \tilde{v}_1 \otimes \tilde{w}_1+\cdots+{\lambda_r}{\gamma_r^2}
     \tilde{v}_r\otimes \tilde{v}_r\otimes \tilde{w}_r
    &=&\cF_{[1:r,1:r,r+1:n]} .
\earay
\end{equation*}
Denote the label sets
\begin{equation}\label{label: J1 J2}
\boxed{
\begin{array}{l}
J_1 \coloneqq \big \{(0,i_1,i_2): \, i_1 \in [r], \, i_2 \in [r+1, n] \big \},\\
J_2 \coloneqq \big \{(i_1,i_2,i_3):\,
      i_1\neq i_2,\, i_1, i_2\in [r], i_3 \in [r+1, n] \big\}.
    \end{array}
    }
\end{equation}
To determine the scalars $\lmd_i, \gamma_i$,
we can solve the linear least squares
\begin{equation}\label{coef-1}
\min\limits_{ (\beta_1,\ldots,\beta_r) } \left\|\cF_{J_1}-
    \sum\limits_{i=1}^r \beta_i
  \cdot \tilde{v}_i  \otimes \tilde{w}_i  \right\|^2 ,
\end{equation}
\begin{equation}\label{coef-2}
    \min\limits_{(\theta_1,\ldots,\theta_r)} \left\|\cF_{J_2}-
     \sum\limits_{k=1}^r\theta_k 
     \cdot (\tilde{v}_k \otimes \tilde{v}_k \otimes \tilde{w}_i)_{J_2} \right\|^2 .
\end{equation}
Let $(\beta_1^*,\ldots,\beta_r^*)$, $(\theta_1^*,\ldots,\theta_r^*)$ be
minimizers of \reff{coef-1} and \reff{coef-2} respectively.
Then, for each $i=1,\ldots,r$, let
\be \label{lam:gam:k}
\lambda_i \, := \, (\beta_i^*)^2/\theta_i^*, \quad
\gamma_i \, := \, \theta_i^*/\beta_i^*.
\ee
For the vectors ($i=1,\ldots,r$)
\[
p_i := \sqrt[3]{\lambda_i}(1,\gamma_i \tilde{v}_i,\tilde{w}_i),
\]
the sum $p_1^{\otimes 3}+ \cdots +p_r^{\otimes 3}$
is a tensor decomposition for $\cF$.
This is justified in the following theorem.

\begin{theorem} \label{thm:GPworks}
Suppose the tensor $\cF$ has the decomposition as in \reff{eq:decom_1 F}.
Assume that the vectors $v_1,\ldots, v_r$ are linearly independent
and the vectors $w_1,\ldots,w_r$ are distinct from each other,
where $v_1,\ldots, v_r, w_1,\ldots,w_r$ are defined as in \reff{vecs:viwi}.
Let $\xi$ be a generically chosen coefficient vector and
let $p_1, \ldots, p_r$ be the vectors produced as above.
Then, the tensor decomposition
$\cF = p_1^{\otimes 3}+ \cdots +p_r^{\otimes 3}$
is unique.
\end{theorem}
\begin{proof}
Since $v_1,\ldots, v_r$ are linearly independent,
the tensor decomposition \reff{eq:decom_1 F} is unique,
up to scalings and permutations. By Theorem~\ref{thm:unique G},
there is a unique generating matrix $G$ for $\cF$ satisfying \reff{eq:generating}.
Under the given assumptions, the equation \reff{solve-G} uniquely determines $G$.
Note that $\xi^T w_1, \ldots, \xi^T w_r$ are the eigenvalues of $N(\xi)$
and $v_1,\ldots, v_r$ are the corresponding eigenvectors.
When $\xi$ is generically chosen, the values of
$\xi^T w_1, \ldots, \xi^T w_r$ are distinct eigenvalues of $N(\xi)$.
So $N(\xi)$ has unique eigenvalue decompositions,
and hence \reff{vi=tdvi:wi=tdwi} must hold, up to a permutation of
$(v_1,\ldots, v_r)$.
Since the coefficient matrices have full column ranks,
the linear least squares problems have unique optimal solutions.
Up to a permutation of $p_1, \ldots, p_r$, it holds that
$
p_i = \sqrt[3]{\lambda_i}
\begin{bmatrix} 1 \\ u_i \end{bmatrix}.
$
Then, the conclusion follows readily.
\end{proof}

The following is the algorithm for computing an incomplete tensor decomposition
for $\cF$ when only its subtensor $\cF_{\Omega}$ is given.
\begin{alg}  \label{algo:iSTD}
(Incomplete symmetric tensor decompositions.)
\begin{itemize}

\item [Input:]
A third order symmetric subtensor ${\cF}_\cI$
and a rank $r =\text{rank}_S(\cF)\le \frac{d}{2}-1$.

\item [1.] Determine the matrix $G$ by solving \reff{solve-G}
for each $\alpha=e_i+e_j \in \bB_1$.

\item [2.]  Let $N(\xi)$ be the matrix as in \reff{N_xi},
for a randomly selected vector $\xi$.
Compute the unit length eigenvectors $\tilde{v}_1,\ldots,\tilde{v}_r$
of $N(\xi)$ and choose $\tilde{w}_i$ as in \reff{w_i}.

\item [3.] Solve the linear least squares \reff{coef-1} and \reff{coef-2}
to get the coefficients $\lambda_i, \gamma_i$ as in \reff{lam:gam:k}.

\item [4.] For each $i=1,\ldots,r$, let
$p_i := \sqrt[3]{ \lambda_i}(1,
\gamma_i \tilde{v}_i, \tilde{w}_i)$.

\item [Output:] The tensor decomposition
$\cF = (p_1)^{\otimes 3}+\cdots+(p_r)^{\otimes 3}$.

\end{itemize}
\end{alg}

The following is an example of applying Algorithm~\ref{algo:iSTD}.

\begin{example}
Consider the same tensor $\cF$ as in Example~\ref{ex-1}.
The monomial sets $\mathscr{B}_0$, $\mathscr{B}_1$ are the same.
The matrices $A_{ij}[\cF]$ and vectors $b_{ij}[\cF]$ are
\[
 A_{13}[\cF]=A_{23}[\cF]=
 \begin{bmatrix}
		-0.8 & 2.8\\ -1.4 & 4
\end{bmatrix}, \quad
b_{13}[\cF]=\begin{bmatrix}1.6\\2.2\end{bmatrix},
b_{23}[\cF]=\begin{bmatrix}-2\\-3.2\end{bmatrix},
\]
\[
A_{14}[\cF] = A_{24}[\cF]=
\begin{bmatrix}
		1 & -0.8\\ -1.4 & 4
\end{bmatrix}, \quad
b_{14}[\cF]=\begin{bmatrix}1.6\\-3.2\end{bmatrix},
b_{24}[\cF]=\begin{bmatrix}-2\\7.6\end{bmatrix},
\]
\[
A_{15}[\cF]=A_{25}[\cF]=
\begin{bmatrix}
		1 & -0.8\\ -0.8 & 2.8
\end{bmatrix}, \quad
b_{15}[\cF] = \begin{bmatrix}2.2\\-3.2\end{bmatrix},
b_{25}[\cF] = \begin{bmatrix}-3.2\\7.6\end{bmatrix}.
\]
Solve \reff{solve-G} to obtain $G$, which is same as in \reff{example-G}.
The matrices $N_3(G),N_4(G),N_5(G)$ are
\begin{align*}
N_3(G)=\begin{bmatrix}
			1/3 & 2/3\\4/3 & -1/3
		\end{bmatrix},\
N_4(G)=\begin{bmatrix}
			4/3 & -1/3\\-2/3 & 5/3
		\end{bmatrix},\
N_5(G)=\begin{bmatrix}
			5/3 & -2/3\\-4/3 & 7/3
		\end{bmatrix}.
\end{align*}
Choose a generic $\xi$, say, $\xi = (3,4,5)$, then
\[
N(\xi)  =
\begin{bmatrix}1/\sqrt{2} & -1/\sqrt{5} \\ 1/\sqrt{2} & 2/\sqrt{5}\end{bmatrix}
\begin{bmatrix} 12 & 0 \\ 0 & 20\end{bmatrix}
\begin{bmatrix}1/\sqrt{2} & -1/\sqrt{5} \\ 1/\sqrt{2} & 2/\sqrt{5}\end{bmatrix}^{-1}.
\]
The unit length eigenvectors are
\[
\tilde{v}_1 = (1/\sqrt{2},1/\sqrt{2}), \quad
\tilde{v}_2=(-1/\sqrt{5},2/\sqrt{5}) .
\]
As in \reff{w_i}, we get the vectors
\[
    w_1 = (1,1,1),\, w_2 = (-1,2,3).
\]
Solving \reff{coef-1} and \reff{coef-2}, we get the scalars
\[
\gamma_1=\sqrt{2}, \quad \gamma_2=\sqrt{5}, \quad
\lambda_1=0.4, \quad  \lambda_2 = 0.6.
\]
This produces the decomposition
$\cF=\lambda_1u_1^{\otimes 3}+\lambda_2u_2^{\otimes 3}$ for the vectors
\[u_1=(1,\gamma_1v_1,w_1)=(1,1,1,1,1,1), \quad u_2=(1,\gamma_2v_2,w_2)=(1,-1,2,-1,2,3). \]
\end{example}

\begin{remark} \label{remark:rank}
Algorithm~\ref{algo:iSTD} requires the value of $r$.
This is generally a hard question. In computational practice,
one can estimate the value of $r$ as follows.
Let $\mbox{Flat}(\cF)\in \bC^{(n+1) \times (n+1)^2}$ be the flattening matrix,
labelled by $(i, (j,k))$ such that
\[
\mbox{Flat}(\cF)_{i,(j,k)} \, = \, \cF_{ijk}
\]
for all $i,j,k =0,1,\ldots, n$.
The rank of $\text{Flat}(\cF)$ equals the rank of $\cF$
when the vectors $p_1,\ldots, p_r$ are linearly independent.
The rank of $\mbox{Flat}(\cF)$ is not available since only the subtensor
$(\cF)_\Omega$ is known. However, we can calculate the ranks of
submatrices of $(\cF)_\Omega$ whose entries are known.
If the tensor $\cF$ as in \reff{eq:decom_1 F} is such that
both the sets $\{v_1,\ldots,v_r\}$ and $\{w_1,\ldots,w_r\}$
are linearly independent, one can see that $\sum_{i=1}^r \lambda_i v_iw_i^T$
is a known submatrix of $\mbox{Flat}(\cF)$ whose rank is $r$.
This is generally the case if $r\le \frac{d}{2}-1$,
since $v_i$ has the length $r $ and $w_i$ has length $d-1-r \ge r$.
Therefore, the known submatrices of $\mbox{Flat}(\cF)$
are generally sufficient to estimate $\rank_S(\cF)$.
For instance, we consider the case $\cF\in\text{S}^3(\mathbb{C}^7)$.
The flattening matrix $\mbox{Flat}(\cF)$ is
\begin{equation}
\begin{bmatrix}
	\ast & \ast & \ast & \ast & \ast & \ast & \ast\\
	\ast & \ast & \cF_{120} & \cF_{130} & \cF_{140} & \cF_{150} & \cF_{160}\\
	\ast & \cF_{210} & \ast & \cF_{230} & \cF_{240} & \cF_{250} & \cF_{260}\\
	\ast & \cF_{310} & \cF_{320} & \ast & \cF_{340} & \cF_{350} & \cF_{360}\\
	\ast & \cF_{410} & \cF_{420} & \cF_{430} & \ast & \cF_{450} & \cF_{460}\\
	\ast & \cF_{510} & \cF_{520} & \cF_{530} & \cF_{540} & \ast & \cF_{560}\\
	\ast & \cF_{610} & \cF_{620} & \cF_{630} & \cF_{640} & \cF_{650} & \ast
\end{bmatrix},
\end{equation}
where each $\ast$ means that entry is not given.
The largest submatrices with known entries are
\[
\begin{bmatrix}
	\cF_{410} & \cF_{420} & \cF_{430}\\
	\cF_{510} & \cF_{520} & \cF_{530}\\
	\cF_{610} & \cF_{620} & \cF_{630}
\end{bmatrix}, \quad
\begin{bmatrix}
	\cF_{140} & \cF_{150} & \cF_{160}\\
	\cF_{240} & \cF_{250} & \cF_{260}\\
	\cF_{340} & \cF_{350} & \cF_{360}
\end{bmatrix}.
\]
The rank of above matrices generally equals $\rank_S(\cF)$
if $r\le \frac{d}{2}-1 = 2.5$.
\end{remark}

\section{Tensor approximations and stability analysis}
\label{sc:errSTD}

In some applications, we do not have the subtensor $\cF_\Omega$ exactly but
only have an approximation $\widehat{\cF}_\Omega$ for it.
The Algorithm~\ref{algo:iSTD} can still
provide a good rank-$r$ approximation for $\cF$
when it is applied to $\widehat{\cF}_\Omega$.
We define the matrix $A_{ij}[\widehat{\cF}]$ and the vector $b_{ij}[\widehat{\cF}]$
in the same way as in \reff{eq:A,b}, for each $\af = e_i+e_j \in \mathscr{B}_1$.
The generating matrix $G$ for $\cF$
can be approximated by solving the linear least squares
\begin{equation}\label{solve G ls}
    \min_{g \in \bC^r } \quad
     \| A_{ij}[\widehat{\cF}] \cdot g - b_{ij}[\widehat{\cF}]\|^2,
\end{equation}
for each $\alpha=e_i+e_j \in \bB_1$.
Let $\widehat{G}(:,e_i+e_j)$ be the optimizer of the above and
$\widehat{G}$ be the matrix consisting of all such $\widehat{G}(:,e_i+e_j)$.
Then $\widehat{G}$ is an approximation for $G$.
For each $l=r+1,\ldots,n$, define the matrix $N_l(\widehat{G})$
similarly as in \reff{eq:N}.
Choose a generic vector $\xi = (\xi_{r+1}, \ldots, \xi_{n})$ and let
\begin{equation}\label{N_xi^est}
\widehat{N}(\xi) \, \coloneqq \,\xi_{r+1} N_{r+1}(\widehat{G})
+\cdots +\xi_{n}N_{n}(\widehat{G}).
\end{equation}
The matrix $\widehat{N}(\xi)$ is an approximation for $N(\xi)$.
Let $\hat{v}_1,\ldots,\hat{v}_r$
be unit length eigenvectors of $\widehat{N}(\xi)$.
For $k =1,\ldots,r$, let
\begin{equation}\label{w_i^est}
 \hat{w}_k \, := \, \big(
 (\hat{v}_k)^H N_{r+1}(\widehat{G})\hat{v}_k,\ldots,
        (\hat{v}_k)^H N_n(\widehat{G}) \hat{v}_k \big) .
\end{equation}
For the label sets $J_1,J_2$ as in \reff{label: J1 J2},
the subtensors $\widehat{\cF}_{J_1},\widehat{\cF}_{J_2}$
are similarly defined like $\cF_{J_1},\cF_{J_2}$.
Consider the following linear least square problems
\begin{equation}\label{coef-1 ls}
    \min_{(\beta_1,\ldots,\beta_r)}
    \left\|\widehat{\cF}_{J_1} - \sum_{k=1}^r \beta_k
    \cdot \hat{v}_k\otimes \hat{w}_k \right\|^2 ,
\end{equation}
\begin{equation}\label{coef-2 ls}
    \min\limits_{(\theta_1,\ldots,\theta_r)} \left\|\widehat{\cF}_{J_2} -
    \sum\limits_{k=1}^r\theta_i
    \cdot (\hat{v}_k \otimes \hat{v}_k
     \otimes \hat{w}_k)_{J_2} \right\|^2 .
\end{equation}
Let $(\hat{\beta}_1, \ldots, \hat{\beta}_r)$
and $(\hat{\theta}_1, \ldots, \hat{\theta}_r)$ be their optimizers respectively.
For each $k=1,\ldots,r$, let
\be \label{hat:lam:gm:k}
\hat{\lmd}_k \,:= \, (\hat{\beta}_k)^2/\hat{\theta}_k, \quad
\hat{\gamma}_k \, := \, \hat{\theta_k}/\hat{\beta}_k .
\ee
This results in the tensor approximation
\[
\cF \approx ( \hat{p}_1)^{\otimes 3}+\cdots+( \hat{p}_r )^{\otimes 3},
\]
for the vectors $\hat{p}_k := \sqrt[3]{ \hat{\lambda}_k}(1,
\hat{\gamma}_k\hat{v}_k, \hat{w}_k)$.
The above may not give an optimal tensor approximation.
To get an improved one,
we can use $\hat{p}_1,\ldots,\hat{p}_r$ as starting points to
solve the following nonlinear optimization
\begin{equation}\label{prob:nls-p}
\min_{(q_1,\ldots,q_r)} \left\| \left(\sum_{k=1}^r (q_k)^{\otimes 3}-
\widehat{\cF}\right)_{\cI} \right\|^2.
\end{equation}
The minimizer of the optimization \reff{prob:nls-p} is denoted as
$(p_1^{*},\ldots,p_r^{*})$.

Summarizing the above, we have the following algorithm
for computing a tensor approximation.

\begin{alg}  \label{algo:decomposition}
(Incomplete symmetric tensor approximations.)
\begin{itemize}

\item [Input:]
A third order symmetric subtensor $\widehat{\cF}_\cI$
and a rank $r\le \frac{d}{2}-1$.

\item [1.] Find the matrix $\widehat{G}$ by solving \reff{solve G ls}
for each $\alpha=e_i+e_j \in \bB_1$.

\item [2.] Choose a generic vector and
let $\widehat{N}(\xi)$ be the matrix as in \reff{N_xi^est}.
Compute unit length eigenvectors $\hat{v}_1,\ldots,\hat{v}_r$
for $\widehat{N}(\xi)$ and define $\hat{w}_i$ in \reff{w_i^est}.

\item [3.] Solve the linear least squares \reff{coef-1 ls}, \reff{coef-2 ls}
to get the coefficients $\hat{\lambda}_i, \hat{\gamma}_i$.

\item [4.] For each $i=1,\ldots,r$, let
$\hat{p}_i := \sqrt[3]{ \hat{\lambda}_i}(1,
\hat{\gamma}_i \hat{v}_i,\hat{w}_i)$.
Then $(\hat{p}_1)^{\otimes 3}+\cdots+(\hat{p}_r)^{\otimes 3}$
is a tensor approximation for $\widehat{\cF}$.

\item [5.] Use $\hat{p}_1,\ldots, \hat{p}_r$ as starting points to
solve the nonlinear optimization \reff{prob:nls-p}
for an optimizer $(p_1^{*},\ldots,p_r^{*})$.

\item [Output:] The tensor approximation
$(p_1^{*})^{\otimes 3}+\cdots+(p_r^{*})^{\otimes 3}$
for $\widehat{\cF}$.

\end{itemize}
\end{alg}

When $\widehat{\cF}$ is close to $\cF$, Algorithm~\ref{algo:decomposition}
also produces a good rank-$r$ tensor approximation for $\cF$.
This is shown in the following.

\begin{theorem}\label{thm:decom error}
Suppose the tensor $\cF = (p_1)^{\otimes 3}+\cdots+(p_r)^{\otimes 3}$,
with $r \le \frac{d}{2}-1$, satisfies the following conditions:
\begin{itemize}
\item [(i)] The leading entry of each $p_i^{}$ is nonzero;

\item [(ii)] the subvectors $(p_1^{})_{2:r+1},\ldots,(p_r^{})_{2:r+1}$
are linearly independent;

\item [(iii)] the subvectors $(p_1)_{[r+2:j,j+2:d]},\ldots,(p_r)_{[r+2:j,j+2:d]}$
are linearly independent for each $j\in [r+1,n]$;

\item [(iv)] the eigenvalues of the matrix $N(\xi)$ in \reff{N_xi}
are distinct from each other.
\end{itemize}
Let $\hat{p}_i, p_i^{*}$ be the vectors produced by
Algorithm~\ref{algo:decomposition}.
If the distance $\epsilon := \|(\widehat{\cF}-{\cF})_{\cI} \|$ is small enough,
then there exist scalars $\hat{\tau}_i, \tau_i^{*}$ such that
\[
(\hat{\tau}_i)^3 = (\tau_i^{*})^3=1, \quad
\|\hat{\tau}_i \hat{p}_i- p_i\| = O(\epsilon), \quad
\|\tau_i^{*}{p}^{*}_i- p_i\| = O(\epsilon),
\]
up to a permutation of $(p_1, \ldots, p_r)$,
where the constants inside $O(\cdot)$ only depend on $\cF$
and the choice of $\xi$ in Algorithm~\ref{algo:decomposition}.
\end{theorem}
\begin{proof}
The conditions (i)-(ii), by Theorem~\ref{thm:unique G},
imply that there is a unique generating matrix $G$ for $\cF$.
The matrix ${G}$ can be approximated by solving the linear least square problems
\reff{solve G ls}. Note that
\[
\|A_{ij}[\widehat{\cF}]-{A}_{ij}[\cF]\| \leq \eps, \quad
\|b_{ij}[\widehat{\cF}]-{b}_{ij}[\cF]\|\le \epsilon,
\]
for all $\alpha=e_i+e_j\in \mathscr{B}_1$. The matrix ${A}_{ij}[{\cF}]$ can be written as
\[
  {A}_{ij}[{\cF}] = [(p_1)_{[r+2:j,j+2:d]},\ldots,(p_r)_{[r+2:j,j+2:d]}]\cdot [(p_1^{})_{2:r+1},\ldots,(p_r^{})_{2:r+1}]^T.
\]
By the conditions (ii)-(iii), the matrix ${A}_{ij}[{\cF}]$ has full column rank for each $j\in [r+1,n]$ and hence the matrix ${A}_{ij}[\widehat{\cF}]$ has full column rank when $\epsilon$ is small enough.
Therefore, the linear least problems \reff{solve G ls}
have unique solutions and the solution $\widehat{G}$ satisfies that
\[
\|\widehat{G}-{G}\| = O(\epsilon),
\]
where $O(\epsilon)$ depends on $\cF$ (see \cite[Theorem~3.4]{Demmel}).
For each $j=r+1,\ldots,n$, $N_j(\widehat{G})$
is part of the generating matrix $\widehat{G}$, so
\[
 \|N_j(\widehat{G})-{N}_j(G)\|\le \|\widehat{G}-{G}\|
 = O(\epsilon), \quad j=r+1,\ldots,n.
\]
This implies that $\|\widehat{N}(\xi)-{N}(\xi)\|=O(\epsilon) $.
When $\epsilon$ is small enough, the matrix $\widehat{N}(\xi)$
does not have repeated eigenvalues, due to the condition (iv). Thus, the matrix $N(\xi)$ has a set of unit length eigenvectors $\tilde{v}_1,\ldots,\tilde{v}_r$
with eigenvalues $\tilde{w}_1,\ldots,\tilde{w}_r$ respectively, such that
\[
\|\hat{v}_i-\tilde{v}_i\| = O(\epsilon), \quad
\|\hat{w}_i-\tilde{w}_i\| = O(\epsilon) .
\]
This follows from Proposition 4.2.1 in \cite{chatelin2012}.
The constants inside the above $O(\cdot)$
depend only on $\cF$ and $\xi$.
The $\tilde{w}_1,\ldots,\tilde{w}_r $ are scalar multiples of linearly independent vectors $(p_1)_{r+2:d},\ldots,(p_r)_{r+2:d}$ respectively, so
$\tilde{w}_1,\ldots,\tilde{w}_r$ are linearly independent.
When $\epsilon$ is small, ${\hat{w}}_1,\ldots,{\hat{w}}_r$ are linearly independent as well.
The scalars $\hat{\lambda}_i \hat{\gamma}_i $ and $\hat{\lambda}_i(\hat{\gamma}_i)^2 $
are optimizers for the linear least square problems
\reff{coef-1 ls} and \reff{coef-2 ls}.
By Theorem~3.4 in \cite{Demmel}, we have
\[
\|\hat{\lambda}_i\hat{\gamma}_i - {\lambda}_i{\gamma}_i\| = O(\epsilon),\,
\|\hat{\lambda}_i(\hat{\gamma}_i)^2 - {\lambda}_i{\gamma}_i^2\| = O(\epsilon).
\]
The vector $p_i$ can be written as $p_i = \sqrt[3]{\lambda_i}(1,\gamma_i \tilde{v}_i,\tilde{w}_i)$, so we must have $\lambda_i,\gamma_i\neq 0$ due to the condition (ii). Thus, it holds that
\[
\|\hat{\lambda}_i - {\lambda}_i\| = O(\epsilon),\,
\|\hat{\gamma}_i - {\gamma}_i\| = O(\epsilon),
\]
where constants inside $O(\cdot)$ depend only on $\cF$ and $\xi$.
For the vectors $\tilde{p}_i:=\sqrt[3]{\lambda_i}(1,\gamma_i\tilde{v}_i,\tilde{w}_i)$,
we have $\cF = \sum_{i=1}^r \tilde{p}_i^{\otimes 3}$, by Theorem~\ref{thm:GPworks}.
Since $p_1,\ldots,p_r$ are linearly independent by the assumption,
the rank decomposition of $\cF$ is unique up to scaling and permutation.
There exist scalars $\hat{\tau}_i$ such that $(\hat{\tau}_i)^3=1$ and
$\hat{\tau}_i \tilde{p}_i = p_i$, up to a permutation of $p_1,\ldots,p_r$.
For $\hat{p}_i = \sqrt[3]{ \hat{\lambda}_i }(1, \hat{\gamma}_i \hat{v}_i ,\hat{w}_i)$,
we have $\| \hat{\tau}_i \hat{p}_i-p_i\|=O(\epsilon)$,
where the constants in $O(\cdot)$ only depend on $\cF$ and $\xi$.

Since $\| \hat{\tau}_i \hat{p}_i - p_i\|=O(\epsilon )$, we have
$\|(\sum_{i=1}^r (\hat{p}_i)^{\otimes 3}-\cF)_{\cI}\| = O(\epsilon)$.
The $(p_1^{*},\ldots,p_r^{*})$ is a minimizer of \reff{prob:nls-p}, so
\[
\left\|\left(\sum_{i=1}^r (p_i^{*})^{\otimes 3}-\widehat{\cF}\right)_{\cI}\right\| \le
\left\|\left(\sum_{i=1}^r ( \hat{p}_i)^{\otimes 3}-\widehat{\cF}\right)_{\cI}\right\| = O(\epsilon).
\]
For the tensor $\cF^{*}:=\sum_{i=1}^r (p_i^{*})^{\otimes 3}$, we get
\[
\|(\cF^{*}-\cF)_\cI\|\le\|(\cF^{*}-\widehat{\cF})_\cI\|+
\|(\widehat{\cF}-\cF)_\cI\|=O(\epsilon) .
\]
When Algorithm \ref{algo:decomposition} is applied to $(\cF^*)_{\Omega}$,
Step~4 will give the exact decomposition $\cF^{*}=\sum_{i=1}^r (p_i^{*})^{\otimes 3}$.
By repeating the previous argument, we can similarly show that
$\|{p}_i- \tau_i^{*}p_i^{*}\|=O(\epsilon)$ for some $\tau_i^{*}$
such that $(\tau_i^{*})^3=1$, where the constants in $O(\cdot)$ only depend on $\cF$ and $\xi$.
\end{proof}

\begin{remark}
For the special case that $\epsilon=0$,
Algorithm~\ref{algo:decomposition} is the same as Algorithm~\ref{algo:iSTD},
which produces the exact rank decomposition for $\cF$.
The conditions in Theorem \ref{thm:decom error} are satisfied
for generic vectors $p_1,\ldots,p_r$, since $r\le \frac{d}{2}-1$.
The constant in $O(\cdot)$ is not explicitly given in the proof.
It is related to the condition number $\kappa(\cF)$
for tensor decomposition \cite{breiding2018condition}.
It was shown by Breiding and Vannieuwenhoven \cite{breiding2018condition} that
\[
\sqrt{\sum\limits_{i=1}^r\|p_i^{\otimes 3}-\hat{p}_i^{\otimes 3}\|^2}\leq\kappa(\cF)\|\cF-\hat{\cF}\|+c\epsilon^2
\]
for some constant $c$.
The continuity of $\hat{G}$ in $\hat{\cF}$
is implicitly implied by the proof. Eigenvalues and unit eigenvectors of
$\widehat{N}(\xi)$ are continuous in $\hat{G}$.
Furthermore, $\hat{\lambda}_i,\hat{\gamma}_i$ are continuous
in the eigenvalues and unit eigenvectors.
All these functions are locally Lipschitz continuous.
The $\hat{p}_i$ is Lipschitz continuous with respect to $\hat{\cF}$,
in a neighborhood of $\cF$, which also implies an error bound for $\hat{p}_i$.
The tensors $(p_i^*)^{\otimes 3}$ are also locally Lipschitz continuous in
$\widehat{{\cF}}$, as illustrated in \cite{breiding2021condition}.
This also gives error bounds for decomposing vectors $p_i^*$.
We refer to \cite{breiding2018condition,breiding2021condition}
for more details about condition numbers of tensor decompositions.
\end{remark}

\begin{example}
We consider the same tensor $\cF$ as in Example~\ref{ex-1}.
The subtensor $(\cF)_\Omega$ is perturbed to $(\widehat{\cF})_\Omega$.
The perturbation is randomly generated from the Gaussian distribution
$\mathcal{N}(0, 0.01)$.
For neatness of the paper, we do not display $(\widehat{\cF})_\Omega$ here.
We use Algorithm \ref{algo:decomposition} to compute the incomplete tensor approximation.
The matrices $A_{ij}[\widehat{\cF}]$ and vectors $b_{ij}[\widehat{\cF}]$ are given as follows:
{\small
\begin{align*}
  A_{13}[\widehat{\cF}]&=A_{23}[\widehat{\cF}]=\begin{bmatrix}
    -0.8135 & 2.7988\\ -1.3697 & 4.0149 \end{bmatrix},
 &b_{13}[\widehat{\cF}]&=\begin{bmatrix}1.5980\\2.1879\end{bmatrix},\,
 &b_{23}[\widehat{\cF}]&=\begin{bmatrix}-2.0047\\-3.2027\end{bmatrix},\\
 A_{14}[\widehat{\cF}]&=A_{24}[\widehat{\cF}]=\begin{bmatrix}
  1.0277 & -0.8020\\-1.3697 & 4.0149
   \end{bmatrix},
   &b_{14}[\widehat{\cF}]&=\begin{bmatrix}1.5920\\-3.2013\end{bmatrix},\,
   &b_{24}[\widehat{\cF}]&=\begin{bmatrix}-2.0059\\7.5915\end{bmatrix},\\
  A_{15}[\widehat{\cF}]&=A_{25}[\widehat{\cF}]=\begin{bmatrix}
    1.0277 & -0.8020\\-0.8135 & 2.7988
   \end{bmatrix},
   &b_{15}[\widehat{\cF}]&=\begin{bmatrix}2.1993\\-3.2020 \end{bmatrix},\,
   &b_{25}[\widehat{\cF}]&=\begin{bmatrix}-3.1917\\7.6153\end{bmatrix}.\\
 \end{align*}
\noindent}The linear least square problems \reff{solve G ls}
are solved to obtain $\widehat{G}$ and $N_3(\widehat{G}),N_4
(\widehat{G}),N_5(\widehat{G})$, which are
\begin{gather*}
  N_3(\widehat{G})=\begin{bmatrix}
    0.5156 & 0.7208\\1.6132 &   -0.2474
  \end{bmatrix},\
  N_4(\widehat{G})=\begin{bmatrix}
    1.2631 & -0.3665\\-0.6489 &   1.6695
  \end{bmatrix},\\
  N_5(\widehat{G})=\begin{bmatrix}
    1.6131 & -0.6752\\-1.2704 &   2.3517
  \end{bmatrix}.
\end{gather*}
For $\xi=(3,4,5)$, the eigendecomposition of the matrix $\widehat{N}(\xi)$ in \reff{N_xi^est} is
\begin{align*}
  \widehat{N}(\xi) &= \begin{bmatrix} -0.7078 & 0.4470 \\ -0.7064 &  -0.8945\end{bmatrix}\begin{bmatrix}12.0343 & 0 \\ 0 & 20.0786\end{bmatrix}
  \begin{bmatrix}-0.7524 & 0.4499 \\ -0.6588 & -0.8931\end{bmatrix}^{-1}.
\end{align*}
It has eigenvectors $\hat{v}_1=(-0.7078,-0.7064), \hat{v}_2=(0.4470,-0.8945)$.
The vectors $\hat{w}_1,\hat{w}_2$ obtained as in \reff{w_i^est} are
\[
 \hat{w}_1 = (1.2021,0.9918,0.9899),\, \hat{w}_2 = (-1.0389,2.0145,3.0016
).
\]
By solving \reff{coef-1 ls} and \reff{coef-2 ls}, we got the scalars
\[
\hat{\gamma}_1=-1.1990,\, \hat{\gamma}_2=-2.1458, \qquad
\hat{\lambda}_1=0.4521,\, \hat{\lambda}_2=0.6232.
\]
Finally, we got the decomposition
$\hat{\lambda}_1 \hat{u}_1^{\otimes 3}+\hat{\lambda}_2 \hat{u}_2^{\otimes 3}$
with
\begin{gather*}
\hat{u}_1=(1,\hat{\gamma}_1\hat{v}_1,\hat{w}_1)=(1,0.8477,0.8479,1.2021,0.9918,0.9899),\\
\hat{u}_2=(1,\hat{\gamma}_2\hat{v}_2,\hat{w}_2)=(1,-0.9776,1.9102,-1.0389,2.0145,3.0016).
\end{gather*}
They are pretty close to the decomposition of $\cF$.
\end{example}

\section{Learning Diagonal Gaussian Mixture}
\label{sc:learnDGM}

We use the incomplete tensor decomposition or approximation method
to learn parameters for Gaussian mixture models.
The Algorithms \ref{algo:iSTD} and \ref{algo:decomposition}
can be applied to do that.

Let $y$ be the random variable of dimension $d$ for a Gaussian mixture model,
with $r$ components of Gaussian distribution parameters
$(\omega_i,\mu_i,\Sigma_i)$, $i=1,\ldots,r$.
We consider the case that $ r\le \frac{d}{2}-1$.
Let $y_1,\ldots,y_N$ be samples drawn from the Gaussian mixture model.
The sample average
\[
\widehat{M}_1 :=\frac{1}{N} ( y_1 + \cdots + y_N)
\]
is an estimation for the mean $M_1:=\bE[y]=\omega_1 \mu_1 + \cdots + \omega_r \mu_r$.
The symmetric tensor
\[
\widehat{M}_3 := \frac{1}{N} ( y_1^{\otimes 3} + \cdots + y_N^{\otimes 3} )
\]
is an estimation for the third order moment tensor $M_3:=\bE[y^{\otimes 3}]$.
Recall that
$
\mathcal{F} = \sum_{i=1}^r \omega_i \mu_i^{\otimes 3}.
$
When all the covariance matrices $\Sig_i$ are diagonal,
we have shown in \reff{M3-decomp} that
\[
    M_3 = \cF + \sum\limits_{j=1}^d(a_j\otimes e_j\otimes e_j
     +e_j\otimes a_j\otimes e_j+e_j\otimes e_j\otimes a_j).
\]
If the labels $i_1,i_2,i_3$ are distinct from each other,
$(M_3)_{i_1i_2i_3} = (\cF)_{i_1i_2i_3}$.
Recall the label set $\cI$ in \reff{eq:label set}. It holds that
\[
(M_3)_{\Omega} = (\cF)_{\Omega} .
\]
Note that $(\widehat{M}_3)_{\Omega}$ is only an approximation for
$(M_3)_{\Omega}$ and $(\cF)_{\Omega}$, due to sampling errors.
If the rank $r\le \frac{d}{2}-1$, we can apply Algorithm \ref{algo:decomposition}
with the input $(\widehat{M}_3)_{\Omega}$, to compute a rank-$r$ tensor approximation for $\cF$.
Suppose the tensor approximation produced by Algorithm~\ref{algo:decomposition} is
\[
\cF \approx (p_1^{*})^{\otimes 3} + \cdots + (p_r^{*})^{\otimes 3}.
\]
The computed $p_1^{*},\ldots,p_r^{*}$ may not be real vectors, even if
$\cF$ is real. When the error
$\epsilon:=\|(\cF-\widehat{M}_3)_\Omega\|$ is small,
by Theorem~\ref{thm:decom error}, we know
\[
\|\tau_i^{*}p_i^{*}-\sqrt[3]{\omega_i}\mu_i\| \,= \,O(\epsilon)
\]
where $(\tau_i^{*})^3=1$.
In computation, we can choose $\tau_i^{*}$ such that $(\tau_i^{*})^3=1$ and
the imaginary part vector $\text{Im}(\tau_i^{*}p_i^{*})$
has the smallest norm. It can be done by checking the imaginary part of
$\tau_i^{*}p_i^{*}$ one by one for
\[
\tau_i^{*} = 1, \,\, -\frac{1}{2}+\frac{\sqrt{-3}}{2}, \,\,
-\frac{1}{2}-\frac{\sqrt{-3}}{2} .
\]
Then we get the real vector
\[
\hat{q}_i  \,:= \, \text{Re}(\tau_i^{*}p_i^{*}).
\]
It is expected that
$\hat{q}_i  \approx \sqrt[3]{\omega_i} \mu_i$. Since
\[
M_1 = \omega_1 \mu_1 + \cdots + \omega_r \mu_r \approx
    \omega_1^{2/3} \hat{q}_1 + \cdots + \omega_r^{2/3} \hat{q}_r,
\]
the scalars $\omega_1^{2/3},\ldots,\omega_r^{2/3}$
can be obtained by solving the linear least squares
\begin{equation}   \label{solve-w}
\min\limits_{(\beta_1,\ldots,\beta_r)\in\mathbb{R}^r_+} \, \left\|
\widehat{M}_1- \sum_{i=1}^r  \beta_i \hat{q}_i  \right\|^2.
\end{equation}
Let $(\beta_1^*,\ldots,\beta_r^*)$ be an optimizer for the above,
then $\hat{\omega}_i := (\beta_i^*)^{3/2}$ is a good approximation for
$\omega_i$ and the vector
\[
\hat{\mu}_i \, := \,  \hat{q}_i / \sqrt[3]{ \hat{\omega}_i }
\]
is a good approximation for $\mu_i$. We may use
\[
\hat{\mu}_i, \quad \big( \sum_{j=1}^r \hat{\omega}_j \big)^{-1}  \hat{\omega}_i,
\quad i=1,\ldots, r
\]
as starting points to solve the nonlinear optimization
\begin{equation}\label{prob:nls-w mu}
\left\{ \begin{array}{cl}
\min\limits_{(\omega_1,\ldots,\omega_r, \mu_1,\ldots, \mu_r)} &
\|\sum_{i=1}^r \omega_i \mu_i-\widehat{M}_1\|^2 +
\|\sum_{i=1}^r \omega_i (\mu_i^{\otimes 3})_\cI-(\widehat{M}_3)_{\cI}\|^2\\
\text{subject to}  & \omega_1 + \cdots +\omega_r = 1, \,
     \omega_1,\ldots,\omega_r \ge 0,
\end{array} \right.
\end{equation}
for getting improved approximations. Suppose an optimizer of the above is
\[
(\omega_1^{*}, \ldots, \omega_r^{*}, \mu_1^{*},\ldots, \mu_r^{*}).
\]

Now we discuss how to estimate the diagonal covariance matrices $\Sigma_i$. Let
\begin{equation}\label{eq: Als}
\mathcal{A}:=M_3-\cF, \quad
\widehat{\mathcal{A}}  :=
\widehat{M}_3-( \hat{q}_1 )^{\otimes 3}-\cdots -( \hat{q}_r )^{\otimes 3} .
\end{equation}
By \reff{M3-decomp}, we know that
\begin{equation}\label{moment-difference}
 \mathcal{A}  = \sum\limits_{j=1}^d(a_j\otimes e_j\otimes e_j+e_j\otimes a_j\otimes e_j+e_j\otimes e_j\otimes a_j),
\end{equation}
where $a_j=\sum\limits_{i=1}^r\omega_i\sigma_{ij}^2\mu_i$ for $j=1,\cdots,d$.
The equation \reff{moment-difference} implies that
\begin{equation}\label{solve-aj}
    (a_j)_j=\frac{1}{3}\mathcal{A}_{jjj},\quad
    (a_j)_i=\mathcal{A}_{jij},
\end{equation}
for $i,j=1,\cdots,d$ and $i\ne j$.
So we choose vectors $\hat{a}_j \in \re^d$ such that
\begin{equation}\label{solve-aj^ls}
    (\hat{a}_j)_j=\frac{1}{3}\widehat{\mathcal{A}}_{jjj},\quad
    (\hat{a}_j)_i=\widehat{\mathcal{A}}_{jij}
    \quad \mbox{for} \quad i \ne j.
\end{equation}
Since $\hat{a}_j\approx \sum\limits_{i=1}^r\omega_i\sigma_{ij}^2\mu_i$,
the covariance matrices $\Sig_i = \mbox{diag}(\sigma_{i1}^2, \ldots, \sigma_{id}^2)$
can be estimated by solving the nonnegative linear least squares
($j=1,\ldots,d$)
\begin{equation}\label{solve-sigma}
\left\{ \baray{cl}
 \min\limits_{(\beta_{1j}, \ldots, \beta_{rj}) } &
      \left\|\hat{a}_j - \sum\limits_{i=1}^r
              \omega^{*}_i\mu^{*}_i \beta_{ij}\right\|^2 \\
 \mbox{subject to} & \beta_{1j} \ge 0,\ldots, \beta_{rj} \ge 0.
\earay \right.
\end{equation}
For each $j$, let $(\beta^*_{1j}, \ldots, \beta^*_{rj})$ be the optimizer for the above.
When $(\widehat{M}_3)_{\Omega}$ is close to $(M_3)_{\Omega}$,
it is expected that $\beta^*_{ij}$ is close to $(\sigma_{ij})^2$.
Therefore, we can estimate the covariance matrices $\Sig_i$ as follows
\begin{equation} \label{Sig_i^opt}
\Sig_i^{*}  \, := \, \mbox{diag}(\beta^*_{i1}, \ldots,  \beta^*_{id}), \quad
(\sigma_{ij}^{*})^2:=\beta^*_{ij}.
\end{equation}
The following is the algorithm for learning Gaussian mixture models.

\begin{alg}  \label{algo:no opt}
(Learning diagonal Gaussian mixture models.)
\begin{itemize}

\item [Input:] Samples $\{y_1,\ldots,y_N\} \subseteq \mathbb{R}^d$
drawn from a Gaussian mixture model and the number $r$
of component Gaussian distributions.

\item [Step~1.] Compute the sample averages
$\widehat{M}_1 := \frac{1}{N} \sum_{i=1}^N y_i$ and
$\widehat{M}_3 :=\dfrac{1}{N}\sum\limits_{i=1}^N y_i^{\otimes 3}$.

\item [Step~2.] Apply Algorithm \ref{algo:decomposition} to the subtensor
$(\widehat{\cF})_{\Omega} := (\widehat{M}_3)_{\Omega}$.
Let $(p_1^{*})^{\otimes 3}+\cdots +(p_r^{*})^{\otimes 3}$
be the obtained rank-$r$ tensor approximation for $\widehat{\cF}$.
For each $i=1,\ldots,r$, let $\hat{q}_i :=\text{Re}(\tau_ip_i^{*})$ where
$\tau_i$ is the cube root of $1$ that minimizes
the imaginary part vector norm $\|\text{Im}(\tau_i p_i^{*})\|$.

\item [Step~3.] Solve \reff{solve-w} to get $\hat{\omega}_1,\ldots,\hat{\omega}_r$ and $\hat{\mu}_i = q_i/\sqrt[3]{ \hat{\omega}_i},i=1,\ldots,r$.

\item [Step~4.] Use the above $\hat{\omega}_i$, $\hat{q}_i$ as initial points to
solve the nonlinear optimization \reff{prob:nls-w mu}
for the optimal $\omega_i^{*},\mu_i^{*},i=1,\ldots, r$.

\item [Step~5.] Get vectors $\hat{a}_1, \ldots, \hat{a}_d$ as in \reff{solve-aj^ls}.
Solve the optimization \reff{solve-sigma} to get optimizers $\beta_{ij}^*$
and then choose $\Sig_i^*$ as in \reff{Sig_i^opt}.

\item [Output:] Component Gaussian distribution parameters
$(\mu^{*}_i,\Sigma^{*}_i,\omega^{*}_i), i=1,\ldots,r$.

\end{itemize}
\end{alg}

The sample averages $\widehat{M}_1, \widehat{M}_3$ can typically be used as good
estimates for the true moments $M_1,M_3$.
When the value of $r$ is not known, it can be determined as in Remark~\ref{remark:rank}.
The performance of Algorithm~\ref{algo:no opt} is analyzed as follows.

\begin{theorem} \label{thm:error no opt}
Consider the $d$-dimensional diagonal Gaussian mixture model with parameters $\{(\omega_i,\mu_i,\Sigma_i):i\in[r]\}$ and $r\le \frac{d}{2}-1$.
Let $\{(\omega^{*}_i,\mu^{*}_i,\Sigma^{*}_i):i\in[r]\}$ be produced by
Algorithm \ref{algo:no opt}. If the distance
$\epsilon :=\max(\|M_3-\widehat{M}_3\|,\|M_1-\widehat{M}_1\|) $ is small enough
and the tensor $\cF=\sum_{i=1}^r\omega_i \mu_i^{\otimes 3}$ satisfies conditions of
Theorem~\ref{thm:decom error}, then
\[
\|\mu_i-\mu^{*}_i\| = O(\epsilon), \|\omega_i-\omega_i^{*}\| = O(\epsilon),
\| \Sigma_{i} - \Sigma^{*}_{i}\| = O(\epsilon),
\]
where the above constants inside $O(\cdot)$ only depend on parameters $\{(\omega_i,\mu_i,\Sigma_i):i\in[r]\}$ and the choice of
$\xi$ in Algorithm~\ref{algo:no opt}.
\end{theorem}
\begin{proof}
For the vectors $p_i:=\sqrt[3]{\omega_i}\mu_i$, we have
$\cF = \sum_{i=1}^r p_i^{\otimes 3}$. Since
\[
\|(\cF-\widehat{\cF})_{\cI}\| = \|(M_3-\widehat{M}_3)_{\cI}\| \le \epsilon
\]
and $\cF$ satisfies conditions of Theorem \ref{thm:decom error}, we know $\|\tau_i^{*}p^{*}_i-p_i\|=O(\epsilon)$ for some $(\tau_i^{*})^3=1$,
by Theorem~\ref{thm:decom error}. The constants inside $O(\epsilon)$
depend on parameters of the Gaussian model and $\xi$. Then, we have $\|\text{Im}(\tau_i^{*}p_i^{*})\|=O(\epsilon)$ since the vectors $p_i$ are real.
When $\epsilon$ is small enough, such $\tau_i^*$ is the $\tau$ in Step~2
of Algorithm \ref{algo:no opt} that minimizes $\|\text{Im}(\tau_ip_i^*)\|$,
so we have
\[
\| \hat{q}_i-p_i\|\le \|\tau_ip^{*}_i-p_i\|=O(\epsilon)
\]
where $\hat{q}_i=\text{Re}(\tau_ip_i^{*})$ is from Step~2.
The vectors $\hat{q}_1,\ldots, \hat{q}_r$ are linearly independent when
$\epsilon$ is small. Thus, the problem \reff{solve-w} has a unique solution and the weights
$\hat{\omega}_i $ can be found by solving \reff{solve-w}. Since
$\|M_1-\widehat{M}_1\|\le \epsilon $ and $\| \hat{q}_i-p_i\|=O(\epsilon)$,
we have $\|\omega_i - \hat{ \omega }_i \|=O(\epsilon)$ (see \cite[Theorem~3.4]{Demmel}).
The mean vectors $\hat{\mu}_i$ are obtained by
$\hat{\mu}_i = \hat{q}_i/\sqrt[3]{ \hat{\omega}_i }$,
so the approximation error is
\[
\|\mu_i - \hat{\mu}_i\|=\|{p}_i/\sqrt[3]{{\omega}_i }-
\hat{q}_i/\sqrt[3]{ \hat{\omega}_i }\| = O(\epsilon) .
\]
The constants inside the above $O(\epsilon)$ depend on parameters
of the Gaussian mixture model and $\xi$.

The problem \reff{prob:nls-w mu} is solved to obtain $\omega^{*}_i$
and $\mu_i^{*}$, so
\[
\left\|\widehat{M_1} - \sum_{i=3}^r \omega_i^{*} \mu_i^{*}\right\|+\left\|\widehat{\cF} -
\sum_{i=1}^r \omega_i^{*} (\mu_i^{*})^{\otimes 3} \right\| = O(\epsilon).
\]
Let $\cF^*:=\sum_{i=1}^r \omega_i^{*} (\mu_i^{*})^{\otimes 3}=\sum_{i=1}^r(\sqrt[3]{\omega_i^{*}}\mu_i^{*})^{\otimes 3}$, then
\[
\|\cF-\cF^*\| \le \|\cF-\hat{\cF}\|+\|\hat{\cF}-\cF^*\| = O(\epsilon).
\]
Theorem \ref{thm:decom error} implies
$\|p_i- \sqrt[3]{\omega_i^{*}}\mu_i^{*}\|=O(\epsilon)$.
In addition, we have
\[
\left\|\widehat{M_1} - \sum_{i=1}^r \omega_i^{*} \mu_i^{*}\right\|=\left\|\widehat{M_1} - \sum_{i=1}^r (\omega_i^{*})^{2/3} \sqrt[3]{\omega_i^{*}}\mu_i^{*}\right\| = O(\epsilon).
\]
The first order moment is $M_1= \sum_{i=1}^r (\omega_i)^{2/3} p_i$.
Since $\|M_1-\hat{M}_1\|=O(\epsilon)$ and $\|p_i- \sqrt[3]{\omega_i^{*}}\mu_i^{*}\|=O(\epsilon)$, it holds that $\|\omega_i^{2/3}-(\omega_i^{*})^{2/3}\|=O(\epsilon)$ by \cite[Theorem~3.4]{Demmel}.
This implies that $\|\omega_i-\omega_i^{*}\|=O(\epsilon)$, so
\[
\|\mu_i-\mu_i^{*}\|=\|p_i/\sqrt[3]{\omega_i}-
(\sqrt[3]{\omega_i^{*}}\mu_i^{*})/\sqrt[3]{\omega_i^{*}}\|=O(\epsilon).
\]
The constants inside the above $O(\cdot)$ only depend on parameters
$\{(\omega_i,\mu_i,\Sigma_i):i\in[r]\}$ and $\xi$.

The covariance matrices $\Sigma_i$ are recovered by solving the linear least squares
\reff{solve-sigma}. In the least square problems, it holds that
$\|\omega_i\mu_i-\omega_i^{*}\mu_i^{*}\|=O(\epsilon)$ and
\[
  \|\mathcal{A}-\widehat{\mathcal{A}}\|\le \|M_3-\widehat{M}_3\|+\|\cF-\sum_{i=1}^r \hat{q}_i^{\otimes 3}\|=O(\epsilon),
\]
where tensors $\mathcal{A},\widehat{\mathcal{A}}$ are defined in \reff{eq: Als}.
When the error $\epsilon$ is small,
the vectors $\omega_i^{*}\mu_1^{*},\ldots,\omega_i^{*}\mu_r^{*}$
are linearly independent and hence \reff{solve-sigma} has a unique solution for each $j$.
By \cite[Theorem~3.4]{Demmel}, we have
\[
\| (\sigma_{ij} )^2 - (\sigma_{ij}^{*})^2\| = O(\epsilon).
\]
It implies that $\|\Sigma_i-\Sigma^{*}_i \| = O(\epsilon)$,
where the constants inside $O(\cdot)$ only depend on parameters
$\{(\omega_i,\mu_i,\Sigma_i):i\in[r]\}$ and $\xi$.
\end{proof}

\section{Numerical Simulations}
\label{sc:num}

This section gives numerical experiments for our proposed methods.
The computation is implemented in {\tt MATLAB} R2019b, on an Alienware personal computer with Intel(R)Core(TM)i7-9700K CPU@3.60GHz and RAM 16.0G.
The {\tt MATLAB} function \texttt{lsqnonlin} is used to solve \reff{prob:nls-p} in Algorithm \ref{algo:decomposition} and the MATLAB function \texttt{fmincon}
is used to solve \reff{prob:nls-w mu} in Algorithm \ref{algo:no opt}.
We compare our method with the classical EM algorithm,
which is implemented by the {\tt MATLAB} function \texttt{fitgmdist}
(\texttt{MaxIter} is set to be $100$ and \texttt{RegularizationValue}
is set to be $0.001$).

First, we show the performance of Algorithm \ref{algo:decomposition}
for computing incomplete symmetric tensor approximations.
For a range of dimension $d$ and rank $r$, we get the tensor
$\cF = (p_1)^{\otimes 3}+\cdots+(p_r)^{\otimes 3}$,
where each $p_i$ is randomly generated
according to the Gaussian distribution in {\tt MATLAB}.
Then, we apply the perturbation
$(\widehat{\cF})_\Omega = (\cF)_\Omega + \mc{E}_{\Omega}$,
where $\mc{E}$ is a randomly generated tensor,
also according to the Gaussian distribution in {\tt MATLAB},
with the norm $\| \mc{E}_{\omega} \|_{\Omega} = \eps$.
After that, Algorithm \ref{algo:decomposition} is applied to the subtensor
$(\widehat{\cF})_\Omega$ to find the rank-$r$ tensor approximation.
The approximation quality is measured by the absolute error and the relative error
\[
\text{abs-error} \coloneqq \|(\cF^{\ast}-\cF)_{\Omega}\|, \quad
\text{rel-error} \coloneqq \frac{\|(\cF^{\ast}-\widehat{\cF})_{\Omega}\|}
{\|(\cF-\widehat{\cF})_{\Omega}\|},
\]
where $\cF^{\ast}$ is the output of Algorithm \ref{algo:decomposition}.
For each case of $(d,r,\eps)$, we generate $100$ random instances.
The min, average, and max relative errors for each dimension $d$ and rank $r$
are reported in the Table~\ref{tensor-result}.
The results show that Algorithm~\ref{algo:decomposition}
performs very well for computing tensor approximations.
{\footnotesize
\begin{table}[htbp]
  \centering
  \caption{The performance of Algorithm~\ref{algo:decomposition}}
  \label{tensor-result}
    \begin{tabular}{ccccccccccc}
    \toprule
     &  & & &$\text{rel-error}$ & &  & &$\text{abs-error}$ & \\
    \cmidrule{4-6} \cmidrule{8-10}
    $d$ & $r$ & $\eps$ & min & average & max & & min & average & max & time\\
    \midrule
    \multirow{4}{*}{20} & 3 & 0.1 &  0.9610  & 0.9731  &  0.9835 & & 0.0141 & 0.0268 & 0.0556 & 0.2687\\
     \cmidrule{4-6} \cmidrule{8-10}
     & 5 & 0.01 & 0.9634 & 0.9700  &  0.9742 & & 0.0019 & 0.0032 & 0.0068 & 0.2392\\
     \cmidrule{4-6} \cmidrule{8-10}
     & 7 & 0.001 & 0.9148  & 0.9373  & 0.9525 & & $2.3\cdot 10^{-4}$ & $3.8\cdot 10^{-4}$ & $6.6\cdot 10^{-4}$ & 0.2638\\
    \midrule
    \multirow{4}{*}{30} & 4 & 0.1 &  0.9816  &  0.9854 & 0.9890 & & 0.0094 & 0.0174 & 0.0533 & 0.4386\\
     \cmidrule{4-6} \cmidrule{8-10}
     & 8 & 0.01 & 0.9634 & 0.9700 & 0.9742 & & 0.0015 & 0.0024 & 0.0060 & 0.7957\\
     \cmidrule{4-6} \cmidrule{8-10}
     & 11 & 0.001 & 0.9501 & 0.9587 & 0.9667 & & $1.8\cdot 10^{-4}$ & $3.0\cdot 10^{-4}$ & $5.7\cdot 10^{-4}$ & 0.8954\\
    \midrule
    \multirow{4}{*}{40} & 6 & 0.1 & 0.9853 &   0.9877  &  0.9904  & & 0.0099 & 0.0146 & 0.0359 & 1.7779\\
     \cmidrule{4-6} \cmidrule{8-10}
     & 10 & 0.01 & 0.9761 & 0.9795 &  0.9820  & & 0.0013 & 0.0020 & 0.0045 & 2.6454\\
     \cmidrule{4-6} \cmidrule{8-10}
     & 15 & 0.001 &0.9653  &  0.9690  &  0.9734 & & $1.7\cdot 10^{-4}$ & $2.6\cdot 10^{-4}$ & $4.8\cdot 10^{-4}$ & 3.6785\\
     \midrule
     \multirow{4}{*}{50} & 7 & 0.1 & 0.9887 & 0.9911 & 0.9925 & & 0.0081 & 0.0128 & 0.0294 & 4.9774\\
     \cmidrule{4-6} \cmidrule{8-10}
     & 13 & 0.01 & 0.9812 & 0.9831 &   0.9854 & & 0.0011 & 0.0018 & 0.0045 & 8.7655\\
     \cmidrule{4-6} \cmidrule{8-10}
     & 18 & 0.001 & 0.9739 & 0.9767 & 0.9792 & & $1.5\cdot 10^{-4}$ & $2.2\cdot 10^{-4}$ & $4.1\cdot 10^{-4}$ & 11.6248\\
     \bottomrule
    \end{tabular}
\end{table}
}

Second, we explore the performance of Algorithm \ref{algo:no opt}
for learning diagonal Gaussian mixture models.
We compare it with the classical EM algorithm, for which
the MATLAB function \texttt{fitgmdist} is used
(\texttt{MaxIter} is set to be 100 and \texttt{RegularizationValue}
is set to be $0.0001$).
The dimensions $d=20,30,40,50,60$ are tested.
Three values of $r$ are tested for each case of $d$.
We generate $100$ random instances of $\{(\omega_i,\mu_i,\Sigma_i):i=1,\cdots,r\}$
for $d\in\{20,30,40\}$, and $20$ random instances for $d\in\{50,60\}$,
because of the relatively more computational time for the latter case.
For each instance, $10000$ samples are generated.
To generate the weights $\omega_1,\ldots, \omega_r$,
we first use the MATLAB function \texttt{randi} to generate a random
$10000-$dimensional integer vector of entries from $[r]$,
then the occurring frequency of $i$ in $[r]$ is used as the weight $\omega_i$.
For each diagonal covariance matrix $\Sigma_i$,
its diagonal vector is set to be the square of a
random vector generated by the MATLAB function \texttt{randn}.
Each sample is generated from one of $r$ component Gaussian distributions,
so they are naturally separated into $r$ groups. Algorithm \ref{algo:no opt}
and the EM algorithm are applied to fit the Gaussian mixture model
to the $10000$ samples for each instance. For each sample,
we calculate the likelihood of the sample to
each component Gaussian distribution in the estimated Gaussian mixture model.
A sample is classified to the $i$th group if its likelihood for the
$i$th component is maximum. The classification accuracy is the rate that
samples are classified to the correct group. In Table~\ref{simu-result}, for each pair $(d,r)$,
we report the accuracy of Algorithm~\ref{algo:no opt}
in the first row and the accuracy of the EM algorithm in the second row.
As one can see, Algorithm~\ref{algo:no opt} performs better than EM algorithm,
and its accuracy isn't affected when the dimensions and ranks increase.
Indeed, as the difference between the dimension $d$ and the rank $r$ increases,
Algorithm~\ref{algo:no opt} becomes more and more accurate.
This is opposite to the EM algorithm. The reason is that
the difference between the number of rows and
the number of columns of $A_{ij}[\cF]$ in \reff{eq:A,b}
increases as $d-r$ becomes bigger,
which makes Algorithm~\ref{algo:no opt} more robust.

{\footnotesize
\begin{table}[htb]
  \centering
  \caption{Comparison between Algorithm \ref{algo:no opt} and EM for simulations}
	\label{simu-result}
    \begin{tabular}{ccccccc}
    \toprule
     & & \multicolumn{2}{c}{$\text{accuracy}$} & & \multicolumn{2}{c}{$\text{time}$} \\
    \cmidrule{3-4}\cmidrule{6-7 }
    $d$ & $r$ & Algorithm \ref{algo:no opt} & EM & & Algorithm \ref{algo:no opt} & EM\\
    \midrule
    \multirow{4}{*}{20} & 3 & 0.9861 & 0.9763 & & 0.8745 & 0.1649\\
     \cmidrule{3-4}\cmidrule{6-7}
     & 5 & 0.9740 & 0.9400 & & 2.3476 & 0.3852\\
     \cmidrule{3-4}\cmidrule{6-7}
     & 7 & 0.9659  & 0.9252 & & 3.4352 & 0.6777\\
    \midrule
    \multirow{4}{*}{30} & 4 &  0.9965 & 0.9684 & & 4.5266 & 0.2959\\
     \cmidrule{3-4}\cmidrule{6-7}
     & 8 & 0.9923 & 0.9277 & & 8.5494 & 0.8525\\
     \cmidrule{3-4}\cmidrule{6-7}
     & 11 & 0.9895 & 0.9219 & & 17.2091 & 1.4106\\
    \midrule
    \multirow{4}{*}{40} & 6 & 0.9990 & 0.9117 & & 18.9160 & 0.6273\\
   \cmidrule{3-4}\cmidrule{6-7}
     & 10 & 0.9981 & 0.8931 & & 28.4161 & 1.2617\\
     \cmidrule{3-4}\cmidrule{6-7}
     & 15 & 0.9971 & 0.9111 & & 69.8013 & 2.0627\\
     \midrule
     \multirow{4}{*}{50} & 7 & 0.9997 & 0.8997 & & 40.6810 & 0.8314\\
     \cmidrule{3-4}\cmidrule{6-7}
     & 13 & 0.9995 & 0.9073 & & 104.7927 & 1.7867\\
     \cmidrule{3-4}\cmidrule{6-7}
     & 18 & 0.9993 & 0.9038 & & 163.2711 & 2.6862\\
     \midrule
     \multirow{4}{*}{60} & 8 & 0.9999 & 0.8874 & & 93.9836 & 1.1266\\
     \cmidrule{3-4}\cmidrule{6-7}
     & 15 & 0.9998 & 0.8632 & & 234.0331 & 2.6435\\
     \cmidrule{3-4}\cmidrule{6-7}
     & 22 & 0.9995 & 0.8929 & & 497.9371 & 3.5527\\
     \bottomrule
    \end{tabular}
\end{table}
}

Last, we apply Algorithm \ref{algo:no opt} to do texture classifications.
We select $8$ textured images of $512\times 512$ pixels from the VisTex database.
We use the MATLAB function \texttt{rgb2gray} to convert them into grayscale version
since we only need their structure and texture information.
Each image is divided into subimages of $32\times 32$ pixels.
We perform the discrete cosine transformation(DCT) on each block of size
$16\times 16$ pixels with overlap of $8$ pixels. Each component of 'Wavelet-like' DCT feature
is the sum of the absolute value of the DCT coefficients in the corresponding sub-block.
So the dimension $d$ of the feature vector extracted from each subimage is $13$.
We use blocks extracted from the first $160$ subimages for training and those from the rest
$96$ subimages for testing.
We refer to \cite{PERMUTER2006695} for more details.
For each image, we apply Algorithm \ref{algo:no opt}
and the EM algorithm to fit a Gaussian mixture model to the image.
We choose the number of components $r$ according to Remark~\ref{remark:rank}.
To classify the test data, we follow the Bayes decision rule that
assigns each block to the texture which maximizes the posteriori probability,
where we assume a uniform prior over all classes \cite{dixit2011adapted}.
The classification accuracy is the rate that a subimage is correctly
classified, which is shown in Table~\ref{table: texture}.
Algorithm~\ref{algo:no opt} outperforms the classical EM algorithm
for the accuracy rates for six of the images.

\begin{figure}[htb]
  \caption{Textures from VisTex}
  \label{textures}
  \centering
  \begin{tabular}{cccc}
  \begin{subfigure}{.2\textwidth}
    \centering
    \includegraphics[width=.8\linewidth]{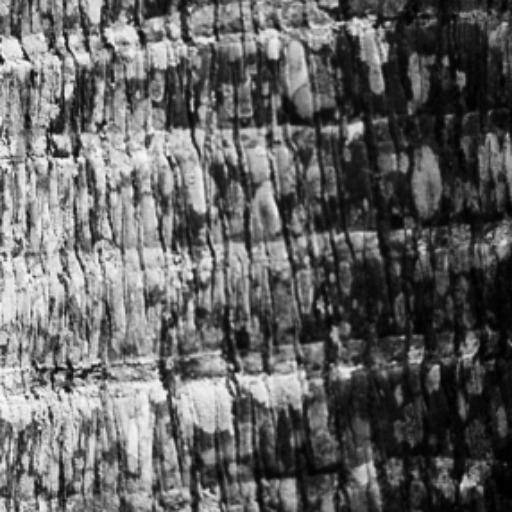}
    \caption*{Bark.0000}
  \end{subfigure}&
  \hfill
  \begin{subfigure}{.2\textwidth}
    \centering
    \includegraphics[width=.8\linewidth]{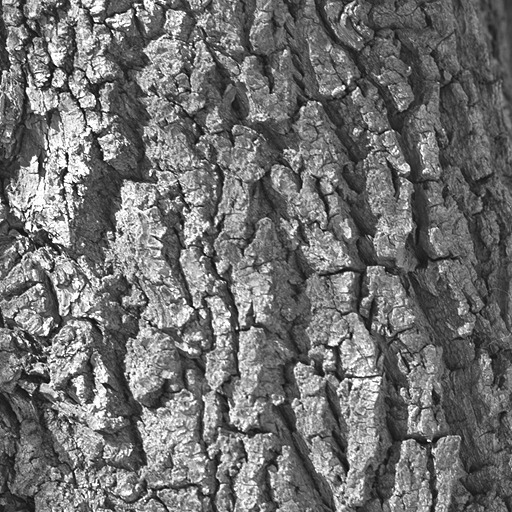}
    \caption*{Bark.0009}
  \end{subfigure}&
  \hfill
  \begin{subfigure}{.2\textwidth}
      \centering
      \includegraphics[width=.8\linewidth]{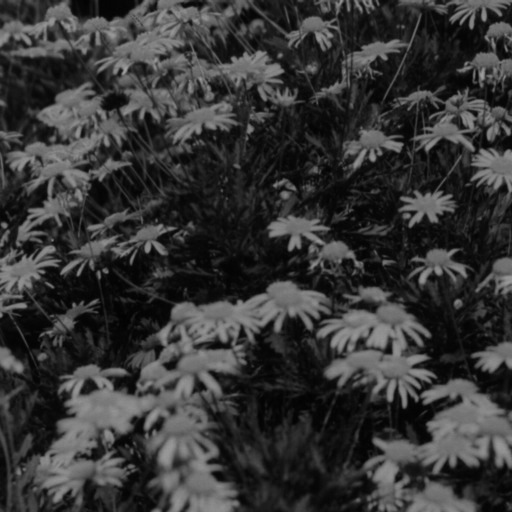}
      \caption*{Flowers.0001}
  \end{subfigure}&
  \hfill
  \begin{subfigure}{.2\textwidth}
      \centering
      \includegraphics[width=.8\linewidth]{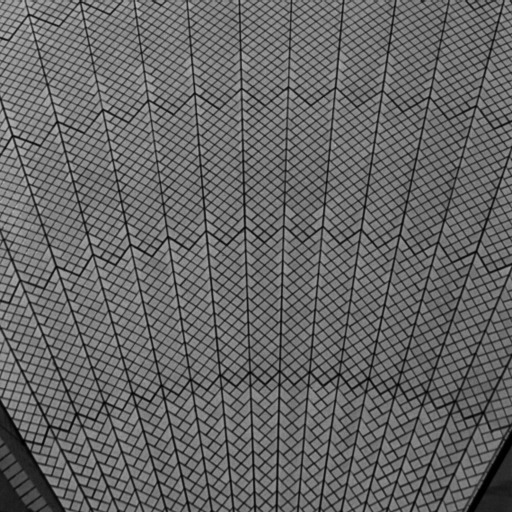}
      \caption*{Tile.0000}
  \end{subfigure}\\
  \begin{subfigure}{.2\textwidth}
      \centering
      \includegraphics[width=.8\linewidth]{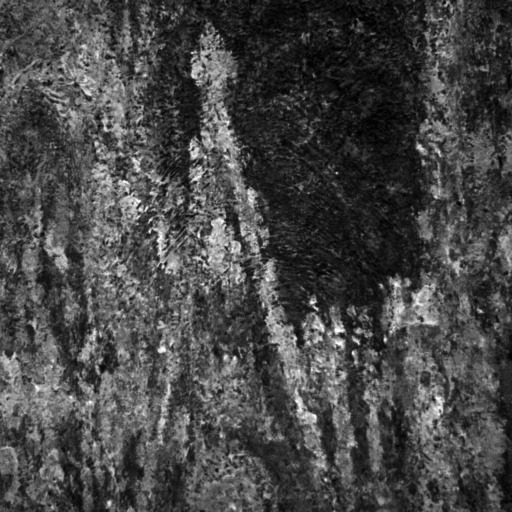}
      \caption*{Paintings.11.0001}
  \end{subfigure}&
  \hfill
  \begin{subfigure}{.2\textwidth}
      \centering
     \includegraphics[width=.8\linewidth]{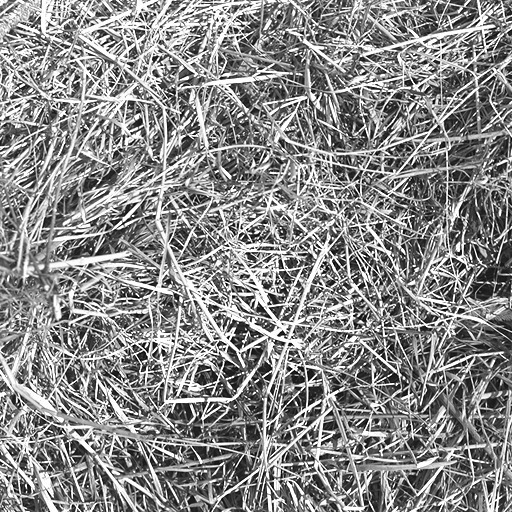}
      \caption*{Grass.0001}
  \end{subfigure}&
  \hfill
  \begin{subfigure}{.2\textwidth}
      \centering
      \includegraphics[width=.8\linewidth]{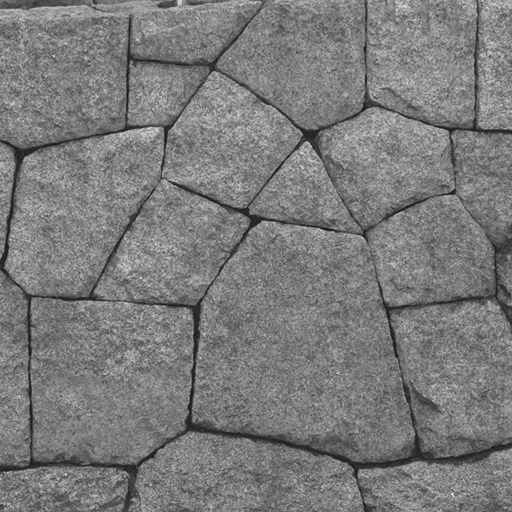}
      \caption*{Brick.0004}
  \end{subfigure}&
  \hfill
  \begin{subfigure}{.2\textwidth}
      \centering
      \includegraphics[width=.8\linewidth]{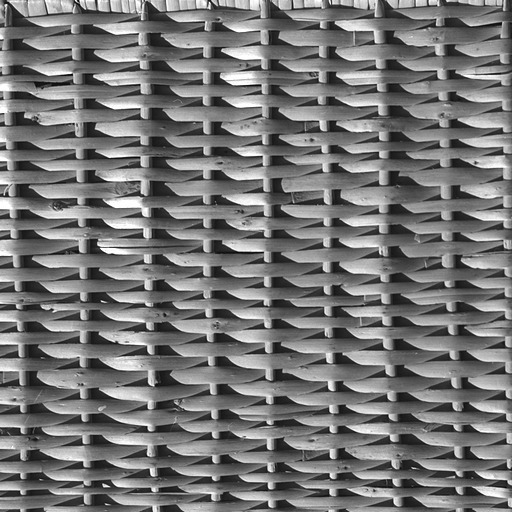}
      \caption*{Fabric.0013}
  \end{subfigure}\\
  \end{tabular}
\end{figure}
\setcounter{figure}{1}


\begin{table}[htb]
	\centering
	\caption{Classification results on $8$ textures}
	\label{table: texture}
	\begin{tabular}{ccc}
	\toprule
	Accuracy & Algorithm \ref{algo:no opt} & EM\\
	\midrule
	Bark.0000 & 0.5376 & 0.8413\\
	Bark.0009 & 0.5107 & 0.7150\\
	Flowers.0001 & 0.8137 & 0.6315\\
	Tile.0000 & 0.8219 & 0.7239\\
	Paintings.11.0001 & 0.8047 & 0.7350\\
	Grass.0001 & 0.9841 & 0.9068\\
	Brick.0004 & 0.9406 & 0.8854\\
	Fabric.0013 & 0.9220 & 0.9048\\
	\bottomrule
	\end{tabular}
\end{table}

\section{Conclusions and Future Work}
\label{sc:con}

This paper gives a new algorithm for learning Gaussian mixture models
with diagonal covariance matrices.
We first give a method for computing incomplete symmetric tensor decompositions.
It is based on the usage of generating polynomials.
The method is described in Algorithm~\ref{algo:iSTD}.
When the input subtensor has small errors,
we can similarly compute the incomplete symmetric tensor approximation,
which is given by Algorithm~\ref{algo:decomposition}.
We have shown in Theorem~\ref{thm:decom error} that
if the input subtensor is sufficiently close to a low rank one,
the produced tensor approximation is highly accurate.
Then unknown parameters for Gaussian mixture models can be recovered
by using the incomplete tensor decomposition method.
It is described in Algorithm~\ref{algo:no opt}.
When the estimations of $M_1$ and $M_3$
are accurate, the parameters recovered by Algorithm \ref{algo:no opt} are also accurate.
The computational simulations demonstrate the good performance of the proposed method.

The proposed methods deals with the case that
the number of Gaussian components is less than one half of the dimension.
How do we compute incomplete symmetric tensor decompositions
when the set $\Omega$ is not like \reff{eq:label set}?
How can we learn parameters for Gaussian mixture models with more components?
How can we do that when the covariance matrices are not diagonal?
They are important and interesting topics for future research work.

\end{document}